\documentclass[12pt]{amsart}
\usepackage{amsmath,amsfonts,amsthm,amssymb,amscd}
\renewcommand{\Re}{\mathop{\rm Re}\nolimits}
\renewcommand{\Im}{\mathop{\rm Im}\nolimits}

\newcommand{\C}{{\mathbb C}}
\newcommand{\R}{{\mathbb R}}

\newcommand{\Z}{{\mathbb Z}}

\newcommand{\N}{{\mathbb N}}

\theoremstyle{plain}
\newtheorem{theorem}{Theorem}[section]
\newtheorem{lemma}[theorem]{Lemma}
\newtheorem{proposition}[theorem]{Proposition}

\theoremstyle{definition}

\theoremstyle{remark}
\newtheorem{remark}[theorem]{Remark}

\def\a{\alpha}

\numberwithin{equation}{section}

\def\a{\alpha}

\def\l{\lambda}

\def\be{\begin{equation}}
\def\ee{\end{equation}}

\title[Energy critical 3d NLS]{Non-dispersive vanishing and blow up at infinity for the energy critical nonlinear Schr\"odinger
equation in $\R^3$}
\author{Cecilia Ortoleva}
\address{LAMA, UMR CNRS 8050, Universit\'e Paris-Est Cr\'eteil, 61, avenue du G\'en\'eral de Gaulle, 94010 Cr\'eteil Cedex, France}
\email{cecilia.ortoleva@math.cnrs.fr}
\author{Galina Perelman}
\address{LAMA, UMR CNRS 8050, Universit\'e Paris-Est Cr\'eteil, 61, avenue du G\'en\'eral de Gaulle, 94010 Cr\'eteil Cedex, France}
\email{galina.perelman@u-pec.fr}

\date{}
\begin{document}
\maketitle
\rightline{{\it Dedicated to the memory of Vladimir Savelievich  Buslaev}}
\begin{abstract}

We consider the energy critical focusing nonlinear Schr\"odinger equation 
$i\psi_t=-\Delta \psi-|\psi|^4\psi $ in $\R^3$, and prove, for any $\nu$ and $\alpha_0$ sufficiently small,
the existence of radial finite energy solutions of the form
$\psi(x,t)=e^{i\alpha(t)}\lambda^{1/2}(t)W(\lambda(t)x)+e^{i\Delta t}\zeta^*+o_{\dot H^1} (1)$ as $t\rightarrow +\infty$, where
$\alpha(t)=\alpha_0\ln t$, $\lambda(t)=t^{\nu}$,
$W(x)=(1+\frac13|x|^2)^{-1/2}$ is the ground state,  and  $\zeta^*$
is arbitrary small in $\dot H^1$.
\end{abstract}

\section{Introduction}
\subsection{Setting of the problem and statement of the result}

In this paper we consider the energy
critical focusing nonlinear Schr\"odinger equation 
\begin{equation}\label{1}\begin{split}
&i\psi_t=-\Delta \psi-|\psi|^4\psi, \quad x\in\R^3,\\
&\psi|_{t=0}=\psi_0\in \dot H^{1}(\R^3).
\end{split}
\end{equation}
Cauchy problem \eqref{1} is locally  well posed and the solutions during their life span satisfy
conservation of energy:
\begin{equation}\label{2}
E(\psi (t))\equiv\int (|\nabla \psi(x,t)|^2 -\frac13 |\psi(x,t)|^6)\, dx=E(\psi_0).
\end{equation}
The problem is energy critical in the sense that both \eqref{1} and \eqref{2}
are invariant with respect to the scaling
$\psi(x,t)\rightarrow \lambda^{1/2}\psi(\lambda x, \lambda^2 t)$, $\lambda\in \R_+$.
For $\dot H^1$ small data one has global existence and scattering.  
In the case of large data blow up may occur. Indeed, the classical virial identity
$${d^2\over dt^2}\int |x|^2|\psi(x,t)|^2 dx=
8\int (|\nabla \psi(x,t)|^2 - |\psi(x,t)|^6)\, dx$$
shows that if $x\psi_0\in L^2(\R^3)$ and $E(\psi_0)<0$, then the solution breaks down in finite time.

Furthermore, 
Eq. \eqref{1} admits an explicit stationary solution (ground state):
$$W(x)=(1+\frac13|x|^2)^{-1/2},\quad \Delta W+W^5=0,$$
so that scattering cannot always occur even for solutions that exist globally in time.

The ground state $W$ is known to play an important role in the dynamics of \eqref{1}.
It was proved by Kenig and Merle \cite{KM} that $E(W)$ is an energy  threshold
for the dynamics in the following sense. If $\psi_0$ is radial and $E(\psi_0)<E(W)$ then\\
(i) the solution of \eqref{1} is global and scatters to zero as a free wave in both directions, provided 
$\|\nabla\psi_0\|_{L^2}<\|\nabla W\|_{L^2}$;\\
(ii) the solution blows up in finite time in both direction, provided $\psi_0\in L^2$ and
$\|\nabla\psi_0\|_{L^2}>\|\nabla W\|_{L^2}$.\\
The behavior of radial  solutions with critical energy $E(\psi_0)=E(W)$ was  classified
by Duyckaerts and Merle \cite{DM}. In this case, in addition to the finite time blow up and scattering to zero (and  $W$ itself), one has the existence of
solutions that converge as $t\rightarrow \infty$  to a rescaled ground state.
In the case of energy slightly greater than  $ E(W)$ the dynamics is expected to be more rich 
and  to include the  solutions that as  $t\rightarrow \infty$ behave like
$e^{i\alpha(t)}\lambda^{1/2}(t)W(\lambda(t)x)$ with 
fairly general $\alpha(t)$ and $\lambda(t)$.
For a closely related model of  the critical wave equation, the existence of this type of solutions 
with $\lambda(t)\rightarrow \infty$ (blow up at infinity) and $\lambda(t)\rightarrow 0$, $t\lambda(t)\rightarrow \infty$ (non-dispersive vanishing)
was recently proved
by Donninger and Krieger \cite{DK}. Our objective in this paper is to obtain an analogous result
for NLS \eqref{1}. More precisely, we prove the following.
\begin{theorem}\label{mth}
There exists $\beta_0>0$ such that 
for any $\nu,\,\alpha_0\in \R$ with  $|\nu|+|\alpha_0|\leq \beta_0$ and any $\delta >0$ 
there exist $T>0$ and a radial solution $ \psi\in C([T,+\infty), \dot H^1\cap \dot H^{2})$ to \eqref{1}
of the form:
\be\label{0.10}
\psi(x,t)=e^{i\alpha(t)}\lambda^{1/2}(t)W(\lambda(t)x)+\zeta(x,t),\ee
where 
$
\lambda(t)=t^{\nu}$, $\alpha(t)=\alpha_0\ln t$, and $\zeta(t)$ verifies:
\be\label{0.12}
\begin{split}
&\|\zeta(t)\|_{\dot H^1\cap \dot H^{2}}\leq \delta, \\
&\|\zeta(t)\|_{L^\infty}\leq C t^{-\frac{1+\nu}{2}},\\
&\|<\lambda(t)x>^{-1}\zeta(t)\|_{L^\infty}\leq C t^{-1-\frac{3}{2}\nu},
\end{split}
\ee
for all $t\geq T$.
The constants $C$ here and below are independent of $\nu, \alpha_0$ and $\delta$.\\
Furthermore, there exists $\zeta^*\in\dot H^{s}$, $\forall s> \frac12-\nu$,
such that, as
$t\rightarrow +\infty$, $\zeta(t)- e^{it\Delta}\zeta^*\rightarrow 0$ in $\dot H^1\cap \dot H^{2}$. 
\end{theorem}
\begin{remark} Theorem \ref{mth} remains valid, in fact, with $\dot H^2$ replaced by $\dot H^{k}$ for any 
$ k\geq 2$ (with $\beta_0$ depending on $k$).
\end{remark}
\begin{remark} The restriction on $\nu$ and $\alpha_0$ that appears in Theorem \ref{mth} seems 
to be technical. One  might expect the same result to be true for any $\nu >-1/2$ and any $\alpha_0\in \R$.
\end{remark}
\begin{remark} The solutions we construct to prove  the theorem belong, in fact, to $\dot H^{\frac12-\nu+}$.
\end{remark}
\begin{remark} Using the techniques  developed in this paper one can prove the existence of radial
finite time blow up solutions
of the form 
$\psi(x,t)=e^{i\alpha(t)}\lambda^{1/2}(t)W(\lambda(t)x)+\zeta(x,t)$,
$\lambda(t)=(T-t)^{-1/2-\nu}$, $\alpha(t)=\alpha_0\ln (T-t)$, 
where $\zeta(t)$ is arbitrary small in $\dot H^{1}\cap \dot H^2$
and $\nu >1$, $\alpha_0\in \R$ can be chosen arbitrarily. For the critical wave equation
an analogous result was proved by Krieger, Schlag, Tataru in \cite{KST}, see also \cite{P}
for a similar construction in the context of the critical Schr\"odinger map equation.

\end{remark}
\subsection{Outline of the paper} 
The paper is organized as follows. In Section 2 we construct (Prop. \ref{p1})
 a
sufficiently good  approximate solution of \eqref{1} very much in the spirit
of \cite{DK}, \cite{KST}, \cite{P}. 
In Section 3 we build up an exact solution
by solving the problem for the small remainder  with zero initial data at infinity, the main technical tool 
of the construction
being some suitable energy type estimates for the linearized evolution.
These estimates are proved in Section 4.

\section{Approximate solutions}\label{s2l}
In this section we prove the following result.
\begin{proposition}\label{p1}
For  any $\nu$ and $\alpha_0$ sufficiently small and any $0<\delta\leq 1$ there exists a radial approximate solution 
$\psi^{ap}\in C^\infty(\R^3,\R_+^*)$
of \eqref{1} such that the following holds for $t \geq T$ with some $T = T(\nu, \alpha_0,\delta) > 0$.\\
\noindent (i) $\psi^{ap}$ has  the form:
$\psi^{ap}(x,t)=e^{i\a(t)}\lambda^{1/2}(t)(W(\lambda(t)x)+\chi^{ap}(\lambda(t)x,t))$,
where $\chi^{ap}(y,t)$, $y=\lambda(t)x$, verifies
\begin{align}
&\|\chi^{ap}(t)\|_{\dot H^k}\leq C\delta^{\nu+k-1/2}t^{-\nu(k-1)}, \quad k=1,2,\label{p1.1.1}\\
&\|\chi^{ap}(t)\|_{L^\infty}\leq Ct^{-(1+2\nu)/2},
\label{p1.1.2}\\
&\||y|^{-1}\chi^{ap}(t)\|_{L^\infty}+ \|\nabla\chi^{ap}(t)\|_{L^\infty}\leq C t^{-1-2\nu},\label{p1.1.3}\\
&\||y|^{-2}\chi^{ap}(t)\|_{L^\infty}+ \||y|^{-1}\nabla_y\chi^{ap}(t)\|_{L^\infty}\leq C  (|\nu|+|\alpha_0|)t^{-1-2\nu},\label{p1.1.31}\\
&\|\nabla^2\chi^{ap}(t)\|_{L^\infty}\leq C  (|\nu|+|\alpha_0|)t^{-1-2\nu}.\label{p1.1.4}
\end{align}
Furthermore, there exists $\zeta^*\in \dot H^{s}$, for any $s> \frac12-\nu$, such that,
as $t\rightarrow +\infty$, $e^{i\a(t)}\lambda^{1/2}(t)\chi^{ap}(\l(t)\cdot, t)- e^{it\Delta}\zeta^*\rightarrow 0$ in $\dot H^1\cap \dot H^{2}$.\\
\noindent (ii) The corresponding error $R=-i\psi^{ap}_t-\Delta \psi^{ap}-|\psi^{ap}|^4\psi^{ap}$
satisfies
\be\label{p1.2}
\|R(t)\|_{\dot H^{k}}
\leq t^{-(2+\frac18)(1+2\nu)+\nu(k+1)},\quad k=0,1,2.\ee
\end{proposition}

\noindent The construction of ${\psi}^{ap}(t)$ will be achieved by considering separately the three regions that correspond to three different space scales: the inner region with the scale $t^{\nu} |x| \lesssim 1$, the self-similar region where $|x| = O(t^{1/2})$, and, finally, the remote region where $|x| = O(t)$. In the inner region the solution will be constructed as a perturbation of the profile $e^{i \alpha_0  \ln t} t^{\nu/2} W(t^{\nu}x)$. The self-similar and remote regions are the regions where the solution is small and  is  described essentially by the linear equation $i \psi_t = -\triangle \psi$. In the self-similar region the profile of the solution will be determined uniquely by the matching conditions coming out from the inner region, while in the remote region the profile remains essentially a free parameter of the construction, only the limiting behavior at the origin is prescribed by the matching procedure.

\subsection{The inner region}
\noindent We start by considering the inner region $0 \leq t^\nu|x|\leq 10 t^{1/2 +\nu -\epsilon_1}$ with $0 < \epsilon_1 < 1/2+\nu$ to be fixed later.
Writing $\psi(x,t)$ as
$ \psi(x, t) = e^{i \alpha(t)}\lambda^{1/2}(t) u (\rho, t)$, $\rho=\lambda(t)|x|$,
we get from \eqref{1}
\begin{equation}    \label{eq_V}
i t^{-2 \nu} u_t-\alpha_0 t^{-(1+ 2\nu)}u+i \nu t^{-(1 +2\nu)} ( \frac{1}{2} +\rho \partial_{\rho}) u
= 
-\triangle u -|u|^4 u.
\end{equation}
Write $u(\rho,t)=W(\rho)+\chi(\rho,t)$. Then $\vec \chi(t)={\chi(t)\choose \bar \chi(t)}$ solves
\begin{equation}    \label{eq-chi}
i t^{-2\nu} \vec \chi_t = H \vec \chi +{\mathcal N}(\chi),
\end{equation}
\noindent where
\begin{equation*}\begin{split}
&H = -\triangle\sigma_3 -3W^4 \sigma_3-2W^4 \sigma_3 \sigma_1,\,\,\,
\sigma_1=\left(\begin{array}{cc}
0 & 1\\
1 & 0
\end{array} \right),\,\,
\sigma_3=\left(\begin{array}{cc}
1 & 0\\
0 & -1
\end{array} \right),\\
&{\mathcal N}(\chi) = \left( \begin{array}{cc}
{N}(\chi)\\
-\overline{{N}(\chi)}\end{array} \right),
\quad N(\chi)=
{N}_0 +{ N}_1(\chi) +{ N}_2(\chi),\\
&{ N}_0 = \alpha_0 t^{-(1 +2\nu)} W  -i\nu t^{-(1 +2\nu)} W_1,\quad W_1(\rho) = ( \frac{1}{2} +\rho \partial_{\rho}) W(\rho)\\
&
{N}_1(\chi) =
\alpha_0  t^{-(1 +2\nu)} \chi -i \nu t^{-(1 +2\nu)} ( \frac{1}{2} +\rho \partial_{\rho}) \chi,\\
&{N}_2(\chi) = -|W +\chi|^4(W +\chi) +W^5 +3W^4 \chi +2W^4 \overline{\chi}.
\end{split}
\end{equation*}

We look for a solution to \eqref{eq-chi} of the form
\begin{equation}    \label{ansatz-chi}
\chi(\rho, t) = \sum_{k = 1}^{\infty} t^{-k (1 +2\nu)} \chi_k(\rho).
\end{equation}
Substituting \eqref{ansatz-chi} into \eqref{eq-chi} and identifying the terms with the same powers of $t$ we get the following system for $\{ \chi_k \}_{k \geq 1}$:
\begin{equation}    \label{syst_chi}
H \vec \chi_k = {\mathcal D}_k, \quad k\geq 1,
\end{equation}
where ${\mathcal D}_k = \left (\begin{array}{cc}{D}_k\\
-\overline{{D}_k}
\end{array} \right)$,
\begin{equation*}\begin{split}
&D_1=-\alpha_0W+i\nu W_1,\\
&D_k=D_k^{(1)} +D_k^{(2)},\quad k\geq 2,
\end{split}
\end{equation*}
$D_k^{(1)}$ and $D_k^{(2)}$ being contributions of $it^{-2\nu}\chi_t-N_1(\chi)$ and $-N_2(\chi)$ respectively:
\begin{equation*}\begin{split}
&{D}_k^{(1)} = -i (1 +2\nu) (k -1) \chi_{k -1} -\alpha_0 \nu \chi_{k -1} +i \nu( \frac{1}{2} +\rho \partial_{\rho}) \chi_{k -1},\\
&
N_2(\chi) = -\sum_{k = 2}^{\infty} t^{-k (1 +2 \nu)} D_k^{(2)}(\rho).
\end{split}
\end{equation*}
Note that $D_k$ depends on $\chi_p$, $1 \leq p \leq k-1$ only:
$$D_k =  D_k (\rho; \chi_p, 1 \leq p \leq k-1).$$
We subject \eqref{syst_chi} to zero initial conditions at $0$:
$
\chi_k(0)=\partial_\rho \chi_k(0)=0$.
\begin{lemma}   \label{lemma-chi}
System \eqref{syst_chi} has a unique solution $\{ \chi_k \}_{k \geq 1}$ verifying:\\
i) for any $k \geq 1$, $\chi_k$ is a $C^{\infty}$ function that has an even Taylor  expansion at $\rho = 0$ that starts at order $2k$;  \\
ii)  as $\rho \rightarrow +\infty$, $\chi_k$, $k \geq 1$, has the following asymptotic expansion 
\begin{equation}\label{chi-at-inf}
\chi_k(\rho) = \sum_{l = 0}^{k} \sum_{j \leq 2k -2l -1} \alpha_{l,j}^{(k)} (\ln \rho)^l \rho^j,
\ee
with some coefficients $\alpha_{l,j}^{(k)}$ verifying $\alpha_{k,2m}^{(k)} = 0$ for all $k, m $.
The asymptotic expansion \eqref{chi-at-inf} can be differentiated any number of times with respect to $\rho$.
\end{lemma}
\begin{proof}
It will be convenient for us to rewrite \eqref{syst_chi} as
\begin{equation}
\label{syst-v}
L_+ v_k^+ = G_k^+, \quad
L_- v_k^- = G_k^-, \quad k\geq1,
\end{equation}
where 
$$\begin{array}{ll}
v_k^+ = \Re \chi_k, & v_k^- = \Im \chi_k,\\
G_k^+ = \Re D_k, & G_k^- = \Im D_k,\\
L_+ = -\triangle -5W^4, & L_- = -\triangle -W^4.
\end{array}$$

\noindent For $k = 1$ \eqref{syst-v} gives
\begin{equation}    \label{syst-v_1}
L_+ v_1^+ = -\alpha_0  W, \quad
L_- v_1^- = \nu W_1.
\end{equation}

\noindent The homogeneous equation $L_{\pm} f = 0$ has two explicit solutions $\Phi_{\pm}$, $\Theta_{\pm}$ given by
\begin{equation}\label{k=0.0}\begin{split}
&\Phi_-(\rho) = W(\rho), \quad \Theta_-(\rho) = \left( 1 +\frac{\rho^2}{3} \right)^{-1/2} \left(
\frac\rho3 -\frac{1}{\rho} \right),\\
&\Phi_+(\rho) = W_1(\rho), \quad \Theta_+(\rho) = -2\left( 1 +\frac{\rho^2}{3} \right)^{-3/2} \left(
\frac{1}{\rho} -2\rho +\frac{\rho^3}{9} \right).
\end{split}\end{equation}
Therefore, solving \eqref{syst-v_1} with zero initial conditions at the origin we obtain
\begin{equation}    \label{v}
\begin{array}{ll}
v_1^+(\rho) =\alpha_0 \int_0^{\rho} s^2 (\Theta_+(\rho) \Phi_+(s) -\Theta_+(s) \Phi_+(\rho)) W(s) ds,\\
v_1^-(\rho) =- \nu \int_0^{\rho} s^2 (\Theta_-(\rho) \Phi_-(s) -\Theta_-(s) \Phi_-(\rho)) W_1(s) ds.
\end{array}
\end{equation}

\noindent Since $W$, $W_1$ are $C^{\infty}$ even functions, $v_1^+$ and $v_1^-$ are also $C^{\infty}$ functions with even Taylor expansion at $\rho = 0$ that  starts at order $2$. Furthermore, the asymptotic expansions of $v_1^+$ and $v_1^-$ as $\rho \rightarrow +\infty$ can be obtained directly from \eqref{v}. As claimed, one has
$$v_1^+(\rho) +iv_1^-(\rho) = \sum_{j \leq 1} \alpha_{0, j}^{(1)} \rho^j +\sum_{j \leq 0} \alpha_{1, j}^{(1)} \rho^{2j -1} \ln \rho, \qquad \textrm{as} \,\, \rho \rightarrow +\infty.$$

We next proceed by induction. Let us consider $k > 1$ and assume that we have found $\chi_i$, $i = 1, \cdots, k-1$, that verify i), ii). Then one can easily check that $D_k$ is an even  $C^{\infty}$ function
with a Taylor series at $0$ starting at order $2(k-1)$ and as $\rho \rightarrow +\infty$, $D_k$ admits
an asymptotic expansion of the form
$$D_k(\rho) = \sum_{l = 0}^{k-1} \sum_{j \leq 2k -2l -3} d_{j, l}^{(k)} (\ln \rho)^l \rho^j +
(\ln \rho)^k \sum_{j \leq -5} d_{j, k}^{(k)} \rho^{j},$$
where
$ d_{-2, k-1}^{(k)}=0$
and $ d_{2m, k}^{(k)}=0,\,\,\forall m$.
Therefore, solving $L_{\pm} v_k^{\pm} = G_k^{\pm}$ with zero conditions at $\rho = 0$ we get a $C^\infty$  even solution $v_k^{\pm} $ which is $O(\rho^{2k})$ at the origin. Finally, the asymptotic expansion at infinity follows directly from the representation
$$v_k^{\pm}(\rho) = -\int_0^{\rho} s^2 (\Theta_{\pm}(\rho) \Phi_{\pm}(s) -\Theta_{\pm}(s) \Phi_{\pm}(\rho)) G_k^{\pm}(s) ds.$$
\end{proof}

\begin{remark}\label{rem-chi}
Clearly, for any $k$, $\chi_k $ is a polynomial with respect to $\alpha_0$ and $\nu$ of the form
$$\chi_k=\sum\limits_{1\leq m+n\leq k}\alpha_0^m\nu^n\chi_{m,n}^k(\rho),$$
where the coefficients $\chi_{m,n}^k$ are $C^{\infty}$ functions of $ \rho$ with  an even Taylor expansion at $ 0$ that starts at order $2k$.
As  $\rho \rightarrow +\infty$, $\chi^k_{m,n}$  admits an  asymptotic expansion of the form \eqref{chi-at-inf}.
\end{remark}

For any $N \geq 2$, define
$$\chi^{(N)}(\rho, t) = \sum_{k = 1}^N t^{-k (1 +2\nu)} \chi_k(\rho).$$
It follows from our construction that  $\chi^{(N)}$ verifies
\begin{equation}\label{inn-0}
\begin{split}
\bigg| \rho^{-k} \partial_{\rho}^l (-i t^{-2\nu} \vec\chi^{(N)}_t +&H\vec\chi^{(N)} +\mathcal{N}(\chi^{(N)})) \bigg| 
\leq \\
&C_{N, l, k}t^{-(N+1)(1 +2\nu) }<\rho>^{2N -1 -l-k},\end{split}
\ee
for any $k,l\in \N$, $k+l\leq 2N$, $0 \leq \rho \leq 10 t^{\frac{1}{2} +\nu -\epsilon_1}$, $t \geq 1$.

Fix $N = 27$, $\epsilon_1 = \frac{1 +2\nu}{27},\,\,$\footnote{This choice  has  no specific meaning here.
To produce an approximate solution with an error verifying \eqref{p1.2} it is sufficient to require
$(2N+3)\varepsilon_1>3(1+2\nu)/2$,
$0<\varepsilon_1<\frac{1+2\nu}{20}$, see \eqref{er-inn} and \eqref{ss-1.5}, \eqref{ss-1.6}.}
and set
\begin{equation*}\begin{split}
&u_{in}^{ap} = W + \chi_{in}^{ap},\quad
\chi_{in}^{ap}=\chi^{(27)}, \\
&\mathcal{R}_{in}= -i t^{-2\nu} \partial_tu^{ap}_{in}-\triangle u^{ap}_{in}+\alpha_0  t^{-1 -2\nu} u^{ap}_{in}
 -i \nu t^{-1 -2\nu}( \frac{1}{2} +\rho \partial_{\rho}) u^{ap}_{in} 
-|u^{ap}_{in}|^4 u^{ap}_{in}.
\end{split}
\end{equation*}

As a direct consequence of Lemma \ref{lemma-chi}  and estimate \eqref{inn-0},  we obtain the following result.
\begin{lemma}   \label{behav-chi} For any $\alpha_0\in \R$ and any $\nu>-\frac12$ there exists
$T=T(\alpha_0, \nu)>0$ such that for $t\geq T$ the following holds.\\
(i) The profile $\chi^{ap}_{in}(t)$ verifies
\begin{align}
&\|\chi_{in}^{ap}\|_{L^\infty(0\leq \rho\leq 10 t^{\frac{1}{2} +\nu -\epsilon_1})}\leq C(|\nu|+|\alpha_0|)t^{-\frac12-\nu},\label{inn-1}\\
&\|\rho^{-k}\partial_\rho^l\chi_{in}^{ap}\|_{L^\infty(0\leq \rho\leq 10 t^{\frac{1}{2} +\nu -\epsilon_1})}\leq C(|\nu|+|\alpha_0|)t^{-1-2\nu},\quad 1\leq k+l\leq 2,\label{inn-2}\\
&\|\rho^{-k}\partial_\rho^l\chi_{in}^{ap}\|_{L^2(\rho^2 d\rho, 0\leq \rho\leq 10 t^{\frac{1}{2} +\nu -\epsilon_1})}\leq C(|\nu|+|\alpha_0|)t^{-(\frac12+\nu)(k+l-\frac12)},\,\,k+l\leq 2.\label{inn-4}
\end{align}
(ii) The error
$\mathcal{R}_{in}(t)$
admits the estimate
\begin{equation}\label{er-inn}
\left\| \rho^{-k} \partial_{\rho}^l \mathcal{R}_{in}(t) \right\|_{L^2(\rho^2 d\rho, 0\leq \rho\leq 10 t^{\frac{1}{2} +\nu -\epsilon_1})}\leq t^{-3(1 +2\nu)/4-\varepsilon_1(2N+1/2)},\quad k+l\leq 2.
\end{equation}
\end{lemma}

\subsection{The self-similar region}
We next consider  the self-similar  region
$\frac{1}{10}t^{-\varepsilon_1}\leq |x|t^{-1/2}\leq 10t^{\varepsilon_2}$,
where $0<\varepsilon_2<1/2$ to be fixed later.
Write 
$\psi(x,t) = e^{i \alpha_0 \ln t} t^{-1/4} w(y, t)$, $ y = t^{-1/2} |x|$.
Then, $w(t)$ solves
\begin{equation}    \label{eq_psi-tilde}
i t w_t= (\mathcal{L}  +\alpha_0)w-|w|^4 w,
\end{equation}
 where $\mathcal{L} = -\triangle +\frac{i}{2} \left( \frac{1}{2} +y \partial_y \right) $.

Note that in the limit $\rho\rightarrow +\infty$, $y\rightarrow 0$ one has, at least, formally
\begin{equation}\label{*}\begin{split}
&t^{\nu/2}(W(\rho)+\sum_{k\geq 1}t^{-k(1+2\nu)}\chi_k(\rho))=\\
&t^{-1/4}\sum_{n\geq 0}\sum_{0\leq l\leq \frac{n}{2}}t^{-\frac14(2n+1)(1+2\nu)}(\ln y+(\frac12+\nu)\ln t)^l
\sum_{k\geq l}\alpha^{(k)}_{l, 2k-n-1}y^{2k-n-1},
\end{split}
\end{equation}
 where $\alpha_{l, j}^{(k)}$, $k \neq 0$, are given by Lemma \ref{lemma-chi} and $\alpha_{l, j}^{(0)}$ come from the expansion of $W(\rho)$ as $\rho\rightarrow \infty$:
$$W(\rho) = \sum_{j \leq 0} \alpha_{0, j}^{(0)} \rho^j,\quad \alpha_{0, 2m}^{(0)} = 0 \,\,\,\forall m\in \Z.$$
Eq. \eqref{*}  suggests
the following ansatz for $w$:
\begin{equation}\label{**}
w(y,t)=\sum_{n\geq 0}\sum_{0\leq l\leq \frac{n}{2}}
t^{-\frac14(2n+1)(1+2\nu)}(\ln y+(\frac12+\nu)\ln t)^lA_{n,l}(y).
\ee
As it will become clear later, to prove Proposition \ref{p1}, it is sufficient to consider only three first 
terms of expansion \eqref{**}. Therefore, we look  for an approximate solution  of the form
\begin{equation*}\begin{split}
w^{ap}_{ss}(y,t)=&t^{-(1+2\nu)/4} A_{0,0}(y) +t^{-3(1+2\nu)/4} A_{1,0}(y) \\
&+t^{-5(1+2\nu)/4}\big(A_{2,0}(y) + (\ln y+(\frac12+\nu)\ln t)A_{2, 1}(y) \big).
\end{split}
\end{equation*}
Substituting this ansatz into the expression $-i t w_t +(\mathcal{L}+\alpha_0)w -|w|^4w$ one gets
\begin{equation}\label{ss-1}\begin{split}
-i t \partial_tw^{ap}_{ss}& +(\mathcal{L}+\alpha_0) w^{ap}_{ss} -|w^{ap}_{ss}|^4 w^{ap}_{ss} =
t^{-(1+2\nu)/4} S_{0,0}(y) +t^{-3(1+2\nu)/4} S_{1,0}(y)\\
& +t^{-5(1+2\nu)/4} (S_{0,0}(y) +
(\ln y+(\frac12+\nu)\ln t)S_{2, 1}(y) ) +S(y, t),
\end{split}
\ee
where
$$\begin{array}{ll}
S_{n, 0}(y) = (\mathcal{L} +\mu_n) A_{n, 0}(y), \quad n = 0, 1,\\
S_{2, 1}(y) = (\mathcal{L} +\mu_2)A_{2, 1}(y),\\
S_{2, 0}(y) = (\mathcal{L} +\mu_2) A_{2, 0}(y) -i\nu A_{2, 1}(y) -\frac{2}{y} \partial_y A_{2, 1}(y) -\frac{A_{2,1}(y)}{y^2} -|A_{0,0}(y)|^4 A_{0,0}(y),\\
S(y, t) = -|w^{ap}_{ss}(y,t)|^4 w^{ap}_{ss}(y,t) +t^{-5(1+2\nu)/4} |A_{0,0}(y)|^4 A_{0,0}(y).
\end{array}$$
Here $ \mu_n=\alpha_0+\frac{i}{4}(2n+1)(1+2\nu)$.

 We require that $S_{n,l} = 0$, $n= 0,1,2$, $l = 0,1$, which means that the corresponding $A_{n,l}$ have to solve
\begin{equation}    \label{syst_A}
\left\{ \begin{array}{ll}
( \mathcal{L} +\mu_n) A_{n, 0}= 0, \quad n= 0, 1,\\
( \mathcal{L} +\mu_2) A_{2, 1} = 0,\\
( \mathcal{L} +\mu_2) A_{2, 0}= i\nu A_{2, 1}+\frac{2}{y} \partial_y A_{2, 1} +\frac{A_{2,1}}{y^2} +|A_{0,0}|^4 A_{0,0}
\end{array} \right..
\end{equation}
\noindent In addition, in order to have the  matching  with the inner region, $A_{n,l}$ have  to satisfy
\begin{equation}\label{match-1}
A_{n,l}(y) = \sum_{k \geq l} \alpha_{l, 2k -n -1}^{(k)} y^{2k -n -1},\quad y\rightarrow 0.
\ee
\begin{lemma}\label{sss-1}
There exists a unique solution of \eqref{syst_A} that as $y\rightarrow 0$ admits an asymptotis expansion of the form
\begin{equation}    \label{A-at-0}
A_{n,l}(y) = \sum_{k \geq l} d_{n, k, l} y^{2k -n -1},
\end{equation}
with $d_{0, 0, 0} = \alpha_{0, -1}^{(0)}$, $d_{1, 1, 0} = \alpha_{0, 0}^{(1)}$ and $d_{2, 1, 0} = \alpha_{0, -1}^{(1)}$.
\end{lemma}
\begin{proof}
First of all note that the equation $(\mathcal{L} +\mu) f = 0$ has a basis of solutions $e_1(y,\mu)$, $e_2(y,\mu)$ such that:\\
(i) $e_1(y,\mu) = \frac{1}{y} +(\mu-\frac i4)\tilde e_1(y,\mu)$, where $\tilde e_1$ is an entire function of $y$ and $\mu$,
odd with respect to $y$;\\
(ii) $e_2$ is a  entire function of $y$ and $\mu$,
even with respect to $y$, and  as $y \rightarrow 0$, $e_2(y,\mu)=1+O(y^2).$

Two first equations of \eqref{syst_A} together with \eqref{A-at-0} give
\be\label{A_{0,0}}
A_{0,0}(y) = \alpha_{0, -1}^{(0)} e_1(y,\mu_0),\quad
A_{1,0}(y) = \alpha_{0, 0}^{(1)} e_2(y,\mu_1).
\ee

We next consider the remaining equations of \eqref{syst_A}. Equation $(\mathcal{L} +\mu_2) A_{2, 1}(y) = 0$ and \eqref{A-at-0} yield
$A_{2,1}(y) = c_0 e_1(y,\mu_2)$, with some constant $c_0$.
Then, for $A_{2,0}$ we have
$( \mathcal{L} +\mu_2) A_{2, 0}=F$,
where 
$$F=c_0(i\nu+\frac{2}{y} \partial_y+y^{-2})e_1(y, \mu_2)+|A_{0,0}|^4 A_{0,0}.$$
As $y\rightarrow 0$, $ F$ has an asymptotic expansion of the form
$$F(y)=\sum_{i\geq -2}\kappa_iy^{2i-1},$$
with some coefficients $\kappa_i$, $\kappa_{-2}$ and $\kappa_{-1}+c_0$ being independent of $c_0$.

Write $A_{2, 0}(y) = -\frac{\kappa_{-2}}{6y^3} +\widetilde{A}_{2,0}(y)$.  Then $\widetilde{A}_{2,0}$ solves
\be\label{***}
( \mathcal{L} +\mu_2) \widetilde A_{2, 0}=\widetilde F,
\ee
where $\widetilde F=F+\frac{\kappa_{-2}}{6}( \mathcal{L} +\mu_2)\frac{1}{y^3}$ has the following asymptotics as $y\rightarrow 0$:
$$\widetilde F(y)=\sum_{i\geq -1}\tilde\kappa_iy^{2i-1},\quad \tilde\kappa_{-1}=\tilde \kappa^{0}_{-1}-c_0,$$
with $\tilde\kappa^{0}_{-1}$ independent of $c_0$.
Take $c_0=\tilde\kappa^{0}_{-1}$. Then Eq. \eqref{***} has a unique solution of the form
$$\widetilde{A}_{2,0}(y) =
\alpha_{0,-1}^{(1)} e_1(y,\mu_2)+ {\rm a}\,\, C^\infty {\rm odd\,\, function}.$$
\end{proof}

\begin{remark} By uniqueness,  $A_{n,l}$
given by Lemma \ref{sss-1} verify matching conditions \eqref{match-1}. Note also that all 
 $A_{n,l}$
are entire
functions of $\alpha_0$ and $\nu$.
\end{remark}

 We next study the behavior of $A_{n,l}$ as $y \rightarrow +\infty$. To this purpose notice that for any  $\mu\in \C$, equation $( \mathcal{L} +\mu)f=0$ has a basis of solutions $f_1(y,\mu)$, $f_2(y,\mu)$ 
such that $yf_1$, $yf_2$ are smooth functions in both variables and as  
$y\rightarrow +\infty$ one has
\be\label{ss-2}
f_1(y,\mu) = y^{-1/2+2i\mu}(1 +O(y^{-2})),\,\,
f_2(y,\mu) = e^{i \frac{y^2}{4}} y^{-5/2-2i\mu} (1 +O({y^{-2}})).
\ee
These  asymptotics are uniform in $\mu$
on  compact subsets of $\C$ and  can be differentiated any number of times with respect to $y$.

Decomposing $A_{1,0}$, $A_{2,0}$, $A_{2,1}$ in the basis $f_1$, $f_2$ one gets
\begin{equation}    \label{A-f}
\begin{array}{ll}
A_{n,0}(y) = d_1^n f_1(y,\mu_n) +d_2^n f_2(y,\mu_n),\quad n=0,1,\\
A_{2,1}(y) = d_1^2 f_1(y,\mu_2) +d_2^2 f_2(y,\mu_2),
\end{array}
\end{equation}
with some  coefficients $d_j^n$, $j = 1, 2$, $n=0, 1, 2$.
As a consequence, as $y\rightarrow +\infty$, one has
\begin{equation}    \label{A-at-inf}
\begin{array}{ll}
A_{0,0}(y)=d_1^0 y^{2i \alpha_0  -1 -\nu} (1 +O(y^{-2}))+d_2^0 e^{iy^2/4} y^{-2i \alpha_0  -2 +\nu} (1 +O(y^{-2})),\\
A_{1,0}(y)=d_1^1 y^{2i \alpha_0-2 -3\nu} (1 +O(y^{-2}))+d_2^1e^{iy^2/4} y^{-2i \alpha_0 -1 +3\nu} (1 +O(y^{-2})),\\
A_{2,1}(y)=d_1^2y^{2i \alpha_0  -3 -5\nu} (1 +O(y^{-2}))+d_2^2e^{iy^2/4} y^{-2i \alpha_0 +5\nu} (1 +O(y^{-2})).
\end{array}
\end{equation}
Asymptotics \eqref{A-at-inf} can be differentiated any number of times with respect to $y$.

Let us now consider $A_{2,0}$ and write it as
\begin{equation}    \label{ansatz-A21}
A_{2,0}(y) = 2 d_1^2 \nu \ln y f_1(y,\mu_2) -2(\nu +1) d_2^2 \ln y f_2(y,\mu_2) +\widehat{A}_{2,0}(y).
\end{equation}
Then $\widehat{A}_{2,0}(y)$ solves
\begin{equation}    \label{eq-A21}
( \mathcal{L} +\mu_2) \widehat{A}_{2,0}= G,
\end{equation}
with $G =d_2^2 G_1 +G_2$, where
\begin{equation*}\begin{split}
&G_1= -d_2^2(1+2\nu)(2y^{-1}\partial_y+y^{-2}-i)f_2(y,\mu_2),\\
&G_2=|A_{0,0}|^4 A_{0,0}+
d_1^2(1+2\nu)(2y^{-1}\partial_y+y^{-2})f_1(y,\mu_2).
\end{split}
\end{equation*}

\noindent It follows from the asymptotics \eqref{ss-2}, \eqref{A-at-inf}
 that $G_j$, $j = 1,2$, has the following behavior 
as $y\rightarrow +\infty$,
\begin{equation*}\begin{split}
&G_1(y) =e^{i  y^2/4} y^{-2i\alpha_0}G_{1,1}(y),\quad
G_2(y)=
\sum_{m= -2}^3 e^{i my^2/4}y^{-2i\alpha_0 \nu (2m-1)} G_{2, m}(y),\\
&\partial_y^l G_{1,1}(y)=O(y^{-2+5\nu-l}),\\
&\partial_y^l G_{2,m}(y)=O(y^{-5-5\nu-|m|(1-2\nu)-l}),\quad  -2\leq m\leq 3,
\end{split}
\end{equation*}
for any $l\geq 0$, provided $\nu$ is sufficiently small.

Integrating \eqref{eq-A21} one gets
\begin{equation}    \label{A-hat}
\widehat{A}_{2,0}(y) = \lambda_1 f_1(y,\mu_2) +\lambda_2 f_2(y,\mu_2) +d_2^2g_1(y) + g_2(y).
\end{equation}
Here $\lambda_i$, $ i=1,2,$ is a constant and
 $g_i$, $i=1,2$, is the solution of $( \mathcal{L} +\mu_2)g_i = G_i$,  with the following behavior as $y\rightarrow +\infty$:
\begin{equation}    \label{g-at-inf}
\begin{array}{lll}
g_1(y)=e^{i  y^2/4} y^{-2i\alpha_0}g_{1,1}(y),\\
g_2(y)=
\sum_{m= -2}^3 e^{i my^2/4}y^{-2i\alpha_0 \nu (2m-1)} g_{2, m}(y),\\
\partial_y^l g_{1,1}(y)=O(y^{-2+5\nu-l}),\\
\partial_y^l g_{2,m}(y)=O(y^{-5-5\nu-m(1-2\nu)-l}),\quad  m=0,1\\
\partial_y^l g_{2,m}(y)=O(y^{-7-5\nu-|m|(1-2\nu)-l}),\quad  m=-2,-1,2,3,
\end{array}
\end{equation}
for any $l\geq 0$.

Denote 
\begin{equation*}\begin{split}
&u^{ap}_{ss}(\rho, t) = t^{-(1 +2\nu)/4} w^{ap}_{ss}(t^{-(1 +2\nu)/2} \rho, t),\\
&\chi^{ap}_{ss}(\rho,t)=u^{ap}_{ss}(\rho, t) -W(\rho),\\
&\mathcal{R}_{ss}(\rho,t)=t^{-5(1+2\nu)/4}S(t^{-(1 +2\nu)/2} \rho, t).
\end{split}
\end{equation*}
The next lemma is  a direct consequence of 
\eqref{match-1}, 
\eqref{ss-2}, \eqref{A-at-inf}, \eqref{ansatz-A21}, \eqref{A-hat} and \eqref{g-at-inf}.

\begin{lemma}   \label{behav-psi-hat}
For any $\alpha_0, \nu  \in \R$ sufficiently small  there exists $T(\alpha_0,\nu)> 0$ such that for $t \geq T(\alpha_0,\nu)$ the following  holds.\\
(i) $\chi^{ap}_{ss}(t)$ verifies
\begin{align}\label{ss-1.1}
&\|\chi_{ss}^{ap}(t)\|_{L^\infty(\frac{1}{10} t^{\frac{1}{2} +\nu -\epsilon_1}\leq \rho\leq 10
 t^{\frac{1}{2} +\nu +\epsilon_2})}\leq C t^{-\frac12-\nu},\\
&\|\rho^{-k}\partial_\rho^l\chi_{ss}^{ap}(t)\|_{L^\infty(\frac{1}{10} t^{\frac{1}{2} +\nu -\epsilon_1}\leq \rho\leq 10
 t^{\frac{1}{2} +\nu +\epsilon_2})}\leq Ct^{-1-2\nu},\,\, k+l=1,\label{ss-1.2}\\
&\|\rho^{-k}\partial_\rho^l\chi_{ss}^{ap}(t)\|_{L^\infty( \frac{1}{10}t^{\frac{1}{2} +\nu -\epsilon_1}\leq \rho\leq 10
 t^{\frac{1}{2} +\nu +\epsilon_2})}\leq C(|\alpha_0|+|\nu|)t^{-1-2\nu},\,\,   k+l=2,\label{ss-1.3}\\
&\|\rho^{-k}\partial_\rho^l\chi_{ss}^{ap}(t)\|_{L^2( \rho^2d\rho,\frac{1}{10}t^{\frac{1}{2} +\nu -\epsilon_1}\leq \rho\leq 10
 t^{\frac{1}{2} +\nu +\epsilon_2})}\leq Ct^{-(1+2\nu)(1-2\varepsilon_2)/4},
\,\, 1\leq k+l\leq 2,\label{ss-1.4}
\end{align}
(ii) The error $\mathcal{R}_{ss}(t)$ admits the estimate
\be\label{ss-1.5}
\|\rho^{-k}\partial_\rho^l\mathcal R_{ss}(t)\|_{L^2( \rho^2d\rho,\frac{1}{10}t^{\frac{1}{2} +\nu -\epsilon_1}\leq \rho\leq 10
 t^{\frac{1}{2} +\nu +\epsilon_2})}\leq Ct^{-(2+\frac14)(1+2\nu)+5\varepsilon_1/2},\,\, 0\leq k+l\leq 2.
\ee
(iii) The difference $u^{ap}_{in}(\rho, t)-u^{ap}_{ss}(\rho,t)$ verifies
\be\label{ss-1.6}
|\partial_\rho^l(u^{ap}_{in}(t)-u^{ap}_{ss}(t))|\leq C\rho^{-2-l}
t^{-(1+2\nu)}(\ln t+t^{3(1+2\nu)/2-(2N+3)\varepsilon_1}),
\ee
for any $l\geq 0$ and $\frac1{10}t^{\frac{1}{2} +\nu -\epsilon_1}\leq \rho\leq 10
 t^{\frac{1}{2} +\nu -\epsilon_1}$.
\end{lemma}

\subsection{The remote region}
We next consider the remote region
$|x| \geq \frac{1}{10}t^{1/2+\varepsilon_2}$. In this region we take as  an approximate solution to \eqref{1} 
the following radial profile:
$$\psi^{ap}_{out}(x, t) = v_1(x, t) +
v_2(x, t) +v_3(x, t),$$
where
$$\begin{array}{ll}
&v_1(x, t) =  e^{i \alpha_0  \ln t} [d_1^0t^{-(1 +\nu)/2} f_1(y,\mu_0) +d_1^1t^{-(2+3\nu)/2} f_1(y,\mu_1)],\quad y=t^{-1/2}|x|,\\
&v_2(x, t)=\Theta_\delta \left( \frac{x}{t} \right) 
 e^{i \alpha_0  \ln t} \left[ d_2^0 t^{-(1 +\nu)/2} f_2(y,\mu_0) +d_2^1t^{-(2 +3\nu)/2} f_2(y,\mu_1) +\right.\\
&+t^{-(3 +5\nu)/2} \left.\left( d_2^2g_1(y)-\big(d_2^2(2\nu+1)\ln\left(\frac{|x|}{t}\right)-\lambda_2\big)f_2(y,\mu_2)
\right) \right],
\end{array}$$
$\Theta_\delta(\xi)=\Theta(\frac\xi\delta) $, 
$\Theta \in C^{\infty}_0(\R^3)$ is radial, $\Theta(\xi) = \left\{ \begin{array}{ll}
1 & \textrm{if} \,\, |\xi| \leq 1\\
0 & \textrm{if} \,\, |\xi| \geq 2
\end{array} \right.$.\\
Finally, $v_3(x,t)$ is given by
$$v_3(x,t)=\frac{e^{i\frac{|x|^2}{4t}}}{t^{5/2}}\hat v_3\left( \frac{x}{t} \right),\quad
\hat v_3=-iz\Delta \Theta_\delta-2i\nabla z\cdot \nabla \Theta_\delta,$$
where
$$z(\xi)=d_2^0|\xi|^{-2i\alpha_0-2+\nu}+d_2^1|\xi|^{-2i\alpha_0-1+3\nu}-(d_2^2(2\nu+1)
\ln|\xi|-\lambda_2)|\xi|^{-2i\alpha_0+5\nu}.$$

It follows from the asymptotics \eqref{ss-2} that for 
$t\geq T$ with some $ T=T(\delta)>0$ and any $ l\geq 0$, one has
\be\label{rem-1}\begin{split}
&|\nabla^lv_1(x,t)|\leq C_l|x|^{-l-1-\nu},\quad \frac{1}{10}t^{1/2+\varepsilon_2}\leq |x|,\\
&|\nabla^lv_2(x,t)|\leq\frac{C_l}{t^{3/2}}\left|\frac{x}{t}\right|^{l-2+\nu},
\quad \frac{1}{10}t^{1/2+\varepsilon_2}\leq |x|\leq 2\delta t.
\end{split}
\ee
Furthermore, $v_2$ can be written as
\be\label{rem-2}\begin{split}
&v_2(x,t)=v_{2,0}(x,t)+v_{2,1}(x,t),\\
&v_{2,0}(x,t)=\frac{e^{i\frac{|x|^2}{4t}}}{t^{3/2}}\Theta_\delta \left( \frac{x}{t} \right) 
z\left(\frac{x}{t} \right),\quad
v_{2,1}(x,t)=\frac{e^{i\frac{|x|^2}{4t}}}{t^{3/2}}\Theta_\delta \left( \frac{x}{t} \right) \hat v_{2,1}(x,t),
\end{split}
\ee
with $\hat v_{2,1}$ verifying, for any  $ l\geq 0$,
\be\label{rem-3}
|\nabla^l\hat v_{2,1}(x,t)|\leq {C_l} t^{3-\nu}|{x}|^{-l-4+\nu},
\quad \frac{1}{10}t^{1/2+\varepsilon_2}\leq |x|\leq 2\delta t.
\ee

We next address $v_3$.  One has
\be\label{rem-4}\begin{split}
&\|\nabla^lv_3(t)\|_{L^\infty(|x|\geq \frac{1}{10}t^{1/2+\varepsilon_2})}\leq
C_lt^{-5/2}\delta^{-4+l+\nu},\\
&\|\nabla^lv_3(t)\|_{L^2(|x|\geq \frac{1}{10}t^{1/2+\varepsilon_2})}\leq
C_lt^{-1}\delta^{-5/2+l+\nu},
\end{split}
\ee
for any $l\geq 0$ and  $t\geq T(\delta)$.

As a direct consequence of estimates \eqref{rem-1}, \eqref{rem-3},  \eqref{rem-4}, one obtains
\be\label{rem-5}\begin{split}
&\|\psi^{ap}_{out}(t)\|_{L^\infty(|x|\geq \frac{1}{10}t^{\frac12+\varepsilon_2})}\leq
Ct^{-(\frac12+\varepsilon_2)(1+\nu)},\\
&\||x|^{-1}\psi^{ap}_{out}(t)\|_{L^\infty(|x|\geq \frac{1}{10}t^{\frac12+\varepsilon_2})}\leq
Ct^{-5/4},\\
&
\|\nabla^l\psi^{ap}_{out}(t)\|_{L^\infty(|x|\geq \frac{1}{10}t^{\frac12+\varepsilon_2})}\leq
Ct^{-5/4},\quad l=1,2,\\
&\|\nabla^l\psi^{ap}_{out}(t)\|_{L^2(|x|\geq \frac{1}{10}t^{\frac12+\varepsilon_2})}\leq
C\delta^{\nu+l-1/2},\quad l=1,2,\\
&\|\nabla^l(\psi^{ap}_{out}(t)-v_{2,0}(t))\|_{L^2(|x|\geq \frac{1}{10}t^{\frac12+\varepsilon_2})}\leq
Ct^{-\frac12(\frac12+\varepsilon_2)(1+2\nu)},\quad l=1,2,\\
&\||x|^{-1}(\psi^{ap}_{out}(t)-v_{2,0}(t))\|_{L^2(|x|\geq \frac{1}{10}t^{\frac12+\varepsilon_2})}\leq
Ct^{-\frac12(\frac12+\varepsilon_2)(1+2\nu)},
\end{split}
\ee
provided $\frac38\leq \varepsilon_2<\frac12$, $\nu$ is sufficiently small
and $t\geq T(\delta).$

Denote 
 $$\psi^{ap}_{ss}(x, t) = e^{i \alpha_0  \ln t} t^{-1/4} {w}_{ss}^{ap}(t^{-1/2}|x|, t),$$
and consider the difference $\psi^{ap}_{ss}(x, t)-\psi^{ap}_{out}(x, t)$. For
$\frac{1}{10}t^{1/2+\varepsilon_2} \leq |x| \leq 10 t^{1/2+\varepsilon_2}$ one has
\begin{equation}    \label{diff-ss-out}
\psi^{ap}_{ss}(x, t) -\psi^{ap}_{out}(x, t) = e^{i \alpha_0 \ln t} 
t^{-(3 +5\nu)/2} ((d_1^2(1+2\nu) \ln|x| +\lambda_1) f_1(y,\mu_2) + g_2(y)) ,
\ee
which together with \eqref{ss-2} and \eqref{g-at-inf} implies that
\begin{equation}    \label{der-diff-ss-out}
|\nabla^l (\psi^{ap}_{out} -\psi^{ap}_{ss})| \leq C_l (|\ln t| t^{-(\frac 12+\varepsilon_2)(3+5\nu+l)}+t^{-(\frac 12+\varepsilon_2)(3+5\nu+1)}),
\end{equation}
 for any $l\geq 0$ and  
$\frac{1}{10}t^{1/2+\varepsilon_2} \leq |x| \leq 10 t^{1/2+\varepsilon_2}$,
provided  $\nu$ is sufficiently small.

We next analyze the error
${R}_{out}(t) = -i \partial_t\psi^{ap}_{out}(t) -\triangle \psi^{ap}_{out}(t) -|\psi^{ap}_{out}(t)|^4 \psi^{ap}_{out}(t)$. It has the form
\begin{equation}\label{R-out}
\begin{split}
{R}_{out}(x, t)& = -\frac{e^{i\frac{|x|^2}{4t}}}{t^{9/2}}\left[t\hat v_{2,1}(x,t)\Delta \Theta_\delta\left( \frac{x}{t} \right)+2t^2\nabla \hat v_{2,1}(x,t)\cdot \nabla\Theta_\delta\left( \frac{x}{t} \right)\right.\\
&\left.
+\Delta\hat v_3\left( \frac{x}{t} \right)\right]
-|\psi^{ap}_{out}|^4 \psi^{ap}_{out}.
\end{split}
\end{equation}
Combined with \eqref{rem-1}, \eqref{rem-3}, \eqref{rem-4}, representation \eqref{R-out}
gives for $\frac38\leq \varepsilon_2<\frac12$ and $\nu$  sufficiently small,
\begin{equation}\label{R-out-est}
\|\nabla^l R_{out}(t)\|_{L^2(|x|\geq \frac{1}{10}t^{1/2+\varepsilon_2})}\leq
Ct^{-\frac94(1+2\nu)},\quad t\geq T(\delta),\quad l=0,1,2.
\ee
\subsection{Proof of Proposition \ref{p1}}
We are now in position to conclude the proof of Prop. \ref{p1}. 
Fix $\varepsilon_2$ such that $\frac38\leq \varepsilon_2<\frac12$ and consider the radial profile $\psi^{ap}(x,t)$ defined by
\begin{equation*}
\begin{split}
\psi^{ap}(x,t)=&\Theta(t^{-1/2+\varepsilon_1}x)\psi^{ap}_{in}(x,t)+
(1-\Theta(t^{-1/2+\varepsilon_1}x))\Theta(t^{-1/2-\varepsilon_2}x)\psi^{ap}_{ss}(x,t)\\
&+
(1-\Theta(t^{-1/2-\varepsilon_2}x))\psi^{ap}_{out}(x,t),\quad x\in \R^3,
\end{split}
\end{equation*}
where
$
\psi^{ap}_{in}(x,t)=e^{i\alpha_0\ln t}t^{\nu/2}u^{ap}_{in}(t^\nu |x|,t)$.
Write $\psi^{ap}$ as $\psi^{ap}(x,t)=e^{i\alpha_0\ln t}t^{\nu/2}(W(y)+\chi^{ap}(y,t))$,
$y=t^\nu x$.
By Lemma \ref{behav-chi} (estimates \eqref{inn-1}, \eqref{inn-2}), Lemma \ref{behav-psi-hat} (estimates \eqref{ss-1.1}, \eqref{ss-1.2}, \eqref{ss-1.3})
and \eqref{rem-5} one has
\begin{align}
&\|\chi^{ap}(t)\|_{L^\infty}\leq Ct^{-(1+2\nu)/2}
\label{pp1.1.2}\\
&\||y|^{-1}\chi^{ap}(t)\|_{L^\infty}+ \|\nabla\chi^{ap}(t)\|_{L^\infty}\leq C t^{-1-2\nu},\label{pp1.1.3}\\
&\||y|^{-2}\chi^{ap}(t)\|_{L^\infty}+ \||y|^{-1}\nabla_y\chi^{ap}(t)\|_{L^\infty}\leq C  (|\nu|+|\alpha_0|)t^{-1-2\nu},\label{pp1.1.31}\\
&\|\nabla^2\chi^{ap}(t)\|_{L^\infty}\leq C  (|\nu|+|\alpha_0|)t^{-1-2\nu}.\label{pp1.1.4}
\end{align}
All the estimates stated in this subsection are valid for $\nu$
sufficiently small and $t\geq T(\alpha_0,\nu,\delta)$.

Futhermore, it follows from  Lemma  \ref{behav-chi} (estimate \eqref{inn-4}),
Lemma \ref{behav-psi-hat} (estimate \eqref{ss-1.3}) and two last inequalities in \eqref{rem-5}
that
\be\label{p1-end-2}\begin{split}
\|\nabla^l\chi^{ap}(t)\|_{L^2(|y|\leq 10t^{1/2+\nu+\varepsilon_2})}\leq Ct^{-(1+2\nu)(1-2\varepsilon_2)/4},\quad l=1, 2,\\
\|\nabla^l(\chi^{ap}(t)-\chi^{ap}_0(t))\|_{L^2(|y|\geq t^{1/2+\nu+\varepsilon_2})}\leq 
Ct^{-(1+2\nu)/4},\quad l=1,2,
\end{split}
\ee
where
$\chi^{ap}_0(y,t)=e^{-i\alpha_0\ln t}t^{-\nu/2}v_{2,0}(t^{-\nu}y,t)$.

Inequalities \eqref{p1-end-2} imply, in particular, 
$$\|\nabla^l\chi^{ap}(t)\|_{L^2(\R^3)}\leq Ct^{-\nu(l-1)}\delta^{\nu+l-1/2},\quad l=1, 2.
$$
Moreover, introducing
$\zeta^*(x)=\pi^{-3/2}e^{3i\pi/4}\int_{\R^3}d\xi e^{ix\cdot \xi}\Theta_\delta(2\xi)z(2\xi)$
and observing that $\zeta^*\in \dot H^s(\R^3)$ for any $s>1/2-\nu$, and
$\|\nabla^l (v_{2,0}-e^{i\Delta t}\zeta^*)\|_{L^2(|x|\geq t^{\gamma})}\rightarrow 0$  as
$t\rightarrow +\infty$
for any $\gamma>\frac{1-2\nu}{3-2\nu}$ and any $l\geq 1$,
one obtains
that
$$e^{i\a(t)}\l^{1/2}(t)\chi^{ap}(\l(t)\cdot, t)- e^{it\Delta}\zeta^*\rightarrow 0 \,\,\, {\rm in}\,\,\,\dot H^1\cap \dot H^{2} \,\,\, {\rm as}\,\,\, t\rightarrow +\infty.$$
 This concludes the proof of the first part of Prop. \ref{p1}.

We next consider the error $R=-i\psi^{ap}_t-\Delta \psi^{ap}-|\psi^{ap}|^4\psi^{ap}$.
It has the form
$$R=E_1+E_2+E_3+E_4.$$
where
\begin{equation*}
\begin{split}
E_1=&i(\frac12-\varepsilon_1)t^{-1}(\psi^{ap}_{in}(x,t)-
\psi^{ap}_{ss}(x,t))\tilde\Theta(t^{-1/2+\varepsilon_1}x)\\
&-2t^{-1/2+\varepsilon_1}(\nabla\psi^{ap}_{in}(x,t)-
\nabla\psi^{ap}_{ss}(x,t))\cdot \nabla\Theta(t^{-1/2+\varepsilon_1}x)\\
&-t^{-1+2\varepsilon_1}(\psi^{ap}_{in}(x,t)-
\psi^{ap}_{ss}(x,t))\Delta\Theta(t^{-1/2+\varepsilon_1}x),\\
E_2=&i(\frac12+\varepsilon_2)t^{-1}(\psi^{ap}_{ss}(x,t)-
\psi^{ap}_{out}(x,t))\tilde\Theta(t^{-1/2-\varepsilon_2}x)\\
&-2t^{-1/2-\varepsilon_2}(\nabla\psi^{ap}_{ss}(x,t)-
\nabla\psi^{ap}_{out}(x,t))\cdot \nabla\Theta(t^{-1/2-\varepsilon_2}x)\\
&-t^{-1-2\varepsilon_2}(\psi^{ap}_{ss}(x,t)-
\psi^{ap}_{out}(x,t))\Delta\Theta(t^{-1/2-\varepsilon_2}x),\\
&\tilde \Theta(\xi)=\xi\cdot\nabla\Theta(\xi),
\end{split}
\end{equation*}
and $E_3$, $E_4$ are given by
 \begin{equation*}
\begin{split}
E_3=&\Theta(t^{-1/2+\varepsilon_1}x)R_{in}(x,t)+
(1-\Theta(t^{-1/2+\varepsilon_1}x))\Theta(t^{-1/2-\varepsilon_2}x)R_{ss}(x,t)\\
&+
(1-\Theta(t^{-1/2-\varepsilon_2}x))R_{out}(x,t),\\
E_4=&\Theta(t^{-1/2+\varepsilon_1}x)(|\psi^{ap}_{in}|^4\psi^{ap}_{in}-|\psi^{ap}|^4\psi^{ap})\\
&+(1-\Theta(t^{-1/2+\varepsilon_1}x))\Theta(t^{-1/2-\varepsilon_2}x)(|\psi^{ap}_{ss}|^4\psi^{ap}_{ss}-|\psi^{ap}|^4\psi^{ap})\\
&+
(1-\Theta(t^{-1/2-\varepsilon_2}x))(|\psi^{ap}_{out}|^4\psi^{ap}_{out}-|\psi^{ap}|^4\psi^{ap}).
\end{split}
\end{equation*}
Here
$$R_{in}(x,t)=e^{i\alpha_0\ln t}t^{5\nu/2}\mathcal R_{in}(t^\nu |x|,t),\quad
R_{ss}(x,t)=e^{i\alpha_0\ln t}t^{5\nu/2}\mathcal R_{ss}(t^\nu |x|,t).$$
First we adress $E_1$. By Lemma \ref{behav-psi-hat} (iii)
we have 
\begin{equation}\label{E-1}
\|E_1\|_{H^2}\leq Ct^{-9(1+2\nu)/4+\nu+5\varepsilon_1/2}\ln t\leq C t^{-(2+\frac3{20})(1+2\nu)}.
\ee
Similarly,
from \eqref{der-diff-ss-out} we get for $E_2$:
\begin{equation}\label{E-2}
\|E_2\|_{H^2}\leq Ct^{-1-(\frac 12+\varepsilon_2)(\frac32+5\nu)}\ln t\leq C t^{-(2+\frac14)(1+2\nu)}.
\ee
Next, we consider $E_3$.
From Lemma \ref{behav-chi} (ii) ,  Lemma \ref{behav-psi-hat} (ii) and 
\eqref{R-out-est} it is apparent that 
\be\label{E-3}
\|E_3\|_{H^2}\leq
Ct^{-\frac94(1+2\nu)+5\varepsilon_1/2}\leq C t^{-(2+\frac3{20})(1+2\nu)}.
\ee
Finally, applying Lemma  \ref{behav-chi} (estimates \eqref{inn-1}, \eqref{inn-2}), Lemma \ref{behav-psi-hat} (estimates \eqref{ss-1.1},  \eqref{ss-1.2},\eqref{ss-1.3},\eqref{ss-1.6}) 
and \eqref{rem-5}, \eqref{der-diff-ss-out}, it is not difficult to check that
\be\label{E-4}
\| E_4\|_{H^2}\leq C t^{-3(1+2\nu)}.
\ee
Combining \eqref{E-1}, \eqref{E-2}, \eqref{E-3}, \eqref{E-4},
we get \eqref{p1.2}, which concludes the proof of Prop. \ref{p1}.
\section{Construction of an exact solution} 
We are now in position to prove Theorem \ref{mth}.
Consider \eqref{1} and write 
$\psi(x,t)=e^{i\alpha_0\ln t}t^{\nu/2}U(y,\tau),$
where $y=t^\nu x$ and $\tau=\frac{t^{1+2\nu}}{1+2\nu}$. Further decomposing
$U$ as 
$$U(y,\tau)=U^{ap}(y,\tau)+f(y,\tau),\quad U^{ap}(y,\tau)=e^{-i\alpha_0\ln t}t^{-\nu/2}\psi^{ap}(x,t),$$
where $\psi^{ap}$ is the approximate solution of \eqref{1} given by Prop. \eqref{p1},
we get the following equation for the remainder $f$
\be\label{eq-f}
i \vec f_{\tau} =\mathcal H(\tau)\vec  f+\mathcal F(f) +r,\quad \vec f={f\choose \bar f},
\ee
where
\begin{equation*}
\begin{split}
&\mathcal H(\tau) =H+\tau^{-1}l, \\
&H = -\triangle\sigma_3 -3W^4 \sigma_3-2W^4 \sigma_3 \sigma_1,
\quad l=\frac{\alpha_0}{2\nu+1}\sigma_3-i\frac{\nu}{2\nu+1}(\frac12+y\cdot \nabla),\\
&\mathcal F(f)={F(f)\choose-\overline{F(f)}},\quad F(f)=F_1(f)+F_2(f)\\
&F_1(f)=\mathcal V_1(\tau)f+\mathcal V_2(\tau) \overline{f},\\
&\mathcal V_1(\tau)=3(W^4-|U^{ap}(\tau)|^4),\quad 
\mathcal V_2(\tau)=2(W^4- (U^{ap}(\tau))^2 |U^{ap}(\tau)|^2),\\
&F_2(f)=-|U^{ap} +f|^4 (U^{ap} +f) +|U^{ap}|^4 U^{ap} +
3 |U^{ap}|^4 f +2 (U^{ap})^2 |U^{ap}|^2 \overline{f},\\
&r = \left( \begin{array}{cc}
\texttt{r}\\
-\overline{\texttt{r}}
\end{array} \right), \quad \texttt{r}(y,\tau) = t^{-5\nu/2}e^{-i\alpha_0\ln t}R(x, t).
\end{split}
\end{equation*}
$R$ being the error
given by Prop. \ref{p1}.
 Note that  by Prop. \ref{p1} one has
\begin{align}
&\|\mathcal V_i(\tau)\|_{W^{2,\infty}(\R^3)}\leq C(|\alpha_0|+|\nu|)\tau^{-1},\quad i=1,2,\label{pot}\\
&\|U^{ap}(\tau)\|_{W^{2,\infty}(\R^3)}\leq C,\label{U-ap}\\
&\|r(\tau)\|_{H^2(\R^3)}\leq C\tau^{-2-\frac18}\label{r},
\end{align}
for any $\tau\geq \tau_0$ with some $\tau_0>0$.

Our intention is to solve \eqref{eq-f} with zero condition at $\tau=+\infty$ by a fix point argument.
To carry out this analysis we  will need  some energy type estimates for the linearized equation
$i \vec f_{\tau} =\mathcal H(\tau)\vec  f$. The required estimates are collected in the next subsection,
their proofs being removed to Section 4.
\subsection{Linear estimates}
We start by recalling some basic spectral properties of the  operator $H$
(a more detailed discussion and the proofs can be found, for example, in \cite{DM}).
Since we are considering only radial solutions, 
we will view  $H$ as an operator on $L^2_{rad}(\R^3;\C^2)$ with domain 
$D(H)=H^2_{rad}(\R^3;\C^3)$. $H$ satisfies the  relations
$$\sigma_3H\sigma_3=H^*,\quad \sigma_1H\sigma_1=-H.$$ 
The essential spectrum of $H$ fills up the real axis. 
The discrete spectrum of $H$ consists of 
 two simple purely imaginary eigenvalues $i\lambda_0$, $-i\lambda_0$, 
$\lambda_0>0$.  The corresponding eigenfunctions
$\zeta_+$, $\zeta_-$ are in $\mathcal S(\R^3)$ and can be chosen in such a way that
$\zeta_-=\sigma_1\zeta_+=\bar \zeta_+$.
Notice also that
$HW{1\choose-1}=HW_1{1\choose 1}=0$.
which means that $H$ has a resonance at zero.

Consider the projection of the  linearized equation $i \vec f_{\tau} =\mathcal H(\tau)\vec  f$
onto the essential spectrum of $H$:
\be\label{lin-eq}
i \vec f_{\tau} =P\mathcal H(\tau)P\vec  f.
\ee
Here $P$ is the spectral projection of $H$ onto the essential spectrum given by
$$P=I-P_+-P_-,\quad P_\pm=\frac{\left<\cdot, \sigma_3 \zeta_\mp\right>}{\left<\zeta_\pm, \sigma_3 \zeta_\mp\right>}\zeta_\pm,$$
$\left<\cdot,\cdot\right>$ is the scalar product in $L^2(\R^3,\C^2)$.

Let $U(\tau,s)$ be the propagator associated to Eq. \eqref{lin-eq}. In Section 4 we prove
the following results.
\begin{proposition}\label{lin}
There exists a constant $C > 0$ such that
$$\| U(\tau, s) f \|_{H^2} \leq C \left( \frac{s}{\tau} \right)^{C(|\alpha_1|+|\nu_1|)} \| f\|_{H^2},$$
for any $s\geq \tau>0$  and any $f\in H^2_{rad}$.
Here $\alpha_1=\frac{\alpha_0}{1+2\nu}$,  $\nu_1=\frac{\nu}{1+2\nu}$.
\end{proposition}
\subsection{Contraction argument}

We now transforme \eqref{eq-f} into a fix point problem.
Rewrite \eqref{eq-f} in the following  integral form
\begin{equation}\label{fix-point}
 f(\tau)= J(f)(\tau),
\ee
where 
\begin{equation*}
\begin{split}
& J(f)(\tau)=J_0(f)(\tau)+J_+(f)(\tau)+J_-(f)(\tau),\\
&J_0(f)(\tau )=i\int_\tau^{+\infty} ds U(\tau,s)P(\mathcal F_1(f(s))+r(s)),\\
&J_+(f)(\tau)=i\int_\tau^{+\infty} dse^{\lambda_0(\tau-s)}P_+(\mathcal F_2(f(s))+r(s)),\\
&J_-(f)(\tau)=-i\int_{\tau_1}^{\tau} dse^{-\lambda_0(\tau-s)}P_-(\mathcal F_2(f(s)) +r(s)),\\
&\mathcal F_1(f) =\mathcal F(f)+ s^{-1}l(P_++P_-)\vec f,\\
&\mathcal F_2(f) =\mathcal F(f) +s^{-1}l\vec f,
\end{split}
\end{equation*}
$\tau_1\geq\max\{ \tau_0,1\}$ to be fixed later (slightly abusing notation we identify in \eqref{fix-point}
$\C^2$  vectors of the form  ${f\choose \bar f}$ with their first component $f$).


Our intention is to view $J$ as a mapping in the space 
$C([\tau_1,+\infty), H^2_{rad})$ equipped with the norm
$\||f|\|=\sup_{\tau\geq \tau_1}\|f(\tau)\|_{H^2}\tau^{1+1/16}$,
and to show that $ J$ is contraction of the unite ball
$\||f|\|\leq 1$ into itself provided $|\alpha_0|+|\nu|$ is sufficiently small and 
$\tau_1$ is chosen sufficiently large.
Indeed, by \eqref{U-ap}, \eqref{pot} one has, for any $f,g\in H^2$ with
$\|f\|_{H^2}\leq 1$, $\|g\|_{H^2}\leq 1$, 
$$\|\mathcal F_1(f)-\mathcal F_1(g)\|_{H^2}\leq C(\|f\|_{H^2}+\|g\|_{H^2}+(|\alpha_0|+|\nu|)\tau^{-1})\|f-g\|_{H^2},$$
$$\|P_\pm(\mathcal F_2(f)-\mathcal F_2(g))\|\leq C(\|f\|_{H^2}+\|g\|_{H^2}+(|\alpha_0|+|\nu|)\tau^{-1})\|f-g\|_{H^2},$$
which together with \eqref{r} and Prop. \ref{lin} gives
$$\||J(f)\||\leq 
\frac 12+C\tau_1^{-1/16},\quad
\||J(f)-J(g)\||\leq (\frac12+C\tau_1^{-1/16})\||f-g\||,$$
for any $f,g\in\{\||h\||\leq 1\}$, provided $|\alpha_0|+|\nu|$ is sufficientlt small.
This means that for $\tau_1$ sufficiently large,
$J$ is a contraction of the unit ball $\||f\||\leq 1$ into itself and consequently,
has a unique fixe point $f$ that satisfies
$$\|f(\tau)\|_{H^2}\leq \tau^{-1-1/16},\quad \forall \tau\geq \tau_1,$$
which together with Prop. \ref{p1} gives Theorem \ref{mth}.

\section{Linearized evolution}  \label{linearized}
In this section we prove Prop. \ref{lin}. The proof will be achieved by combining 
the results of \cite{DM} with a careful spectral analysis of the operator $H$
around zero energy. The section organized as follows. In subsection 1
we consider the operator $H$ as before, restricted to the subspace of radial functions,
and construct a basis of Jost solutions for the equation $H\zeta=E\zeta$. In subsection 2
we study the spectral decomposition of $H$ near $E=0$. In subsection 3 we prove Prop. \ref{lin}
by combining the results of the previous two subsections with the coercivity properties
of $H$ established in \cite{DM}.

\subsection{Solutions to the equation $H\zeta=E\zeta$.}
In this subsection we construct a basis of Jost solutions of the equation
$H\zeta=E\zeta$, $E\in \R$.  Since the subject is completely standard we will only briefly scetch the proofs (see also \cite{BP}, \cite{KS} for a closely related construction in the context of energy
subcitical NLS). 
Recall that
\begin{equation*}\begin{split}
&H=-(\partial_\rho^2+2\rho^{-1}\partial_\rho)\sigma_3+V(\rho),\quad V=
\left(\begin{array}{cc}
V_1& V_2\\
-V_2 & -V_1
\end{array} \right),\\
& V_1(\rho)=-3W^4(\rho),\quad V_2(\rho)=-2W^4(\rho),
\quad 
W(\rho)=(1+\rho^2/3)^{-1/2}.
\end{split}
\end{equation*}

We emphase that $V(\rho)$ is a smooth function of $\rho$ that decays as $\rho^{-4}$
as $\rho\rightarrow \infty$.
Since $\sigma_1H=-H\sigma_1$ it suffices to consider the case $E\geq 0$,
so we write $E=k^2$, $k\geq 0$.
It will be convenient for us  to remove the first derivative in $H$.
Set  $f=\rho\zeta$, then one gets
\be\label{lin-2}
\tilde H f=Ef,\quad \tilde H=-\partial_\rho^2\sigma_3+V(\rho).
\ee
We will consider the operator $\tilde H$ on $\R$, to recover the original radial $\R^3$ problem
it suffices to restrict $\tilde H$ to the subspace of odd functions.

We start by  constructing  the most rapidly decaying solution to \eqref{lin-2}.

\begin{lemma}   \label{lemma-f_3}
For all $k\geq 0$ there exists a real solution
$f_3(\rho, k)$  of the equation
\be\label{lin-3}
\tilde H f=k^2f,
\ee
 such that
$f_3(\rho, k)=e^{-k\rho}\chi_3(\rho,k)$, 
where
$\chi_3$ is $C^{\infty}$ function of $( \rho,k)\in \R\times \R_+^*$ verifying
$\chi_3(\rho,k)={0\choose 1}+a(\rho,k)$,
\be\label{f-3}\begin{split}
&|\partial_\rho^l\partial_k^ma(\rho,k)|\leq C_{l}<\rho>^{-2-l+m}(1+k<\rho>)^{-1-m}, m=0,1,\quad\\
&|\partial_\rho^l\partial_k^2a(\rho,k)|\leq C_{l}<\rho>^{-l}(1+k<\rho>)^{-3}\ln \left(\frac{1}{k<\rho>}+2\right),\\
\end{split}
\ee
for all $\rho\geq 0$, $k>0$ and $l\geq 0$.
\end{lemma}

\begin{proof}
One writes the following integral equation for $\chi_3$ 
$$\chi_3(\rho, k) = {0\choose1}-\int_{\rho}^{+\infty} K(\rho-s, k)\sigma_3 V(s)
\chi_3(s, k) ds,$$
$$K(\xi,k)= \left(\begin{array}{cc}
\frac{\sin k\xi}{k} & 0\\
0 & \frac{\sinh k \xi}{k}
\end{array} \right) e^{k\xi}.$$
The statement of the lemma follows then from the estimate
$$|\partial_k^l K(\xi,k)|\leq C_l\frac {|\xi|^{l+1}}{<k\xi>^{l+1}},\quad \xi\leq 0,\,\, k\geq 0,\,\,l\geq 0$$
and the decay properties of $V$:
$$|\partial^l_\rho V(\rho)|\leq C_l<\rho>^{-4-l},\quad \rho \in \R,\quad l\geq 0,$$
by standard Volterra iterations.
\end{proof}

We next construct the oscillating solutions to Eq. \eqref{lin-3}.
\begin{lemma}   \label{lemma-f_1}
For all $k\geq 0$ there exists a solution
$f_1(\rho, k)$  of Eq.
\eqref{lin-3}
 such that $f_1$ is a smooth 
function of $( \rho,k)\in \R\times \R_+^*$ of the form
$f_1(\rho, k)=e^{ik\rho}({1\choose 0}+b(\rho,k))$,
where $b$ verifies
\be\label{f-1}\begin{split}
&|b(\rho,k)|\leq C(<\rho>^{-2}+ke^{-k\rho}),\\
&|\partial_\rho b(\rho,k)|\leq C(<\rho>^{-3}+k^2e^{-k\rho}),\\
&|\partial_k b(\rho,k)|\leq C(<\rho>^{-1}+<k\rho>e^{-k\rho}),\\
&|\partial_{\rho k}^2 b(\rho,k)|\leq C(<\rho>^{-2}+k<k\rho>e^{-k\rho}),
\end{split}
\ee
for all $\rho\geq 0$, $0\leq k\lesssim 1$. In addition, one has
$$ |\partial_{k}^2 b(\rho,k)|\leq C\ln \left(\frac1k+1\right),$$
for all $0\leq \rho\lesssim 1$, $0<k\lesssim 1$.
\end{lemma}
\begin{proof}
To construct $f_1$ we will reduce the order of the system
\eqref{lin-3} by means of the substitution
$f_1=z_0f_3+z_1{1\choose 0}$. Further setting 
$z_2=z_0^\prime f_{3,2}$, $f_3={f_{3,1}\choose f_{3,2}}$, we  get that
$z={z_1\choose z_2}$ solves
\be\label{sys-z}\begin{split}
&-z_{1}^{\prime\prime}-k^2z_1+V_{11}z_1+V_{12}z_2=0,\\
&-z_2^\prime+kz_2+V_{21}z_1+V_{22}z_2=0.
\end{split}
\ee
Here 
\begin{equation*}\begin{split}
&V_{11}=V_1-V_2\frac{f_{3,1}}{f_{3,2}},\quad V_{12}=\frac{2}{f_{3,2}^2}
(f_{3,1}f_{3,2}^\prime-f_{3,1}^\prime f_{3,2}),\\
&V_{21} = V_2, \quad V_{22} = -\frac{1}{f_{3,2}} (
f_{3,2}^\prime +k f_{3,2}).
\end{split}
\end{equation*}
By Lemma \ref{lemma-f_3}, there exists $R> 0$ independent of $k$, such that the functions
$V_{ij}(\rho,k)$, $i,j=1,2$, are smooth in both variables for $k>0$ and $\rho\geq R$
and verify for all $l\geq 0$, $\rho\geq R$, $k>0$,
\be\label{V}\begin{split}
&|\partial_\rho^l V_{j1}(\rho,k)|\leq C_l<\rho>^{-4-l},\quad j=1,2,\\
&|\partial_\rho^l \partial_kV_{11}(\rho,k)|\leq C_l<\rho>^{-5-l}<k\rho>^{-2},\\
&|\partial_\rho^l \partial_k^2V_{11}(\rho,k)|\leq 
C_l<\rho>^{-4-l}<k\rho>^{-3}\ln\left(\frac{1}{k\rho}+2\right),\\
&|\partial_\rho^l \partial_k^mV_{j2}(\rho,k)|\leq 
C_l<\rho>^{-3-l+m}<k\rho>^{-1-m},\quad j=1,2, \quad m=0,1,\\
&|\partial_\rho^l \partial_k^2V_{22}(\rho,k)|\leq 
C_l<\rho>^{-1-l}<k\rho>^{-3}\ln\left(\frac{1}{k\rho}+2\right),
\end{split}
\ee
Writing for $z$ the  following integral equation
$$z(\rho, k) = 
e^{i k \rho}{1\choose 0}
 -\int_{\rho}^{\infty} \left( \begin{array}{cc}
\frac{\sin k (\rho -s)}{k} & 0\\
0 & e^{-k (s -\rho)}
\end{array} \right) \left( \begin{array}{cc}
V_{11} & V_{12}\\
V_{21} & V_{22}
\end{array} \right)z(s, k) ds,$$
and taking into account  \eqref{V},
one proves easily   the existence of a smooth solution  satisfying
\be\label{z}\begin{split}
&|\partial_\rho^l\partial_k^m(e^{-ik\rho}z_1-1)|+<\rho >
|\partial_\rho^l\partial_k^m(e^{-ik\rho}z_2)|
\leq C_l<\rho>^{-2-l+m}<k\rho>^{-1-m},
\,\, m=0,1,\\
&|\partial_\rho^n\partial_k^2(e^{-ik\rho}z_1-1)|+
|\partial_\rho^n\partial_k^2(e^{-ik\rho}z_2)|
\leq C\ln \left(\frac{1}{k\rho}+2\right),
\quad  n=0,1,
\end{split}
\ee
for all $\rho\geq R$, $k>0$, $l\geq 0$.

To reconstruct $f_1$, we set 
$$z_0(\rho, k) = \int_R^{\rho} \frac{z_2(s, k)}{f_{3, 2}(s, k)} ds -\int_{R}^{+\infty} \frac{z_2(s, 0)}{f_{3, 2}(s, 0)} ds.$$
Then, for $\rho \geq R$, the statement of Lemma \ref{lemma-f_1} follows directly  from \eqref{z} and Lemma \ref{lemma-f_3}. To cover the case $x\leq R$ one can invoke the Cauchy problem with initial data at $\rho=R$.
\end{proof}

Note that since $k^2 \in \R$, $f_2(\cdot, k) =
\overline{f_1(\cdot, k)}$ is also a solution of \eqref{lin-3}.

\begin{remark}\label{remark}
Recall that the equation $\tilde H f=0$ has a basis of explicit solutions  
$\rho\Phi_\pm(\rho){1\choose \pm1},\quad \rho\Theta_\pm(\rho){1\choose \pm1}$,
with $\Phi_\pm$, $\Theta_\pm$ given by \eqref{k=0.0}.
Comparing the behavior of $\rho\Phi_\pm$,  $\rho\Theta_\pm$,
with the asymptotics of $f_1(\rho,0)$, $f_3(\rho,0)$,
one gets 
\begin{equation}    \label{k=0.1}
f_1(\rho,0)=\frac12\rho(\xi_0(\rho)+\xi_1(\rho)),\quad f_3(\rho,0)=\frac12\rho(\xi_1(\rho)-\xi_0(\rho)),
\ee
where
$\xi_0=\frac{1}{\sqrt 3}W{1\choose -1}$, 
$\xi_1=-\frac{2}{\sqrt 3}W_1{1\choose 1}$.

\end{remark}
Next, we construct  an  exponentially growing solution at $+\infty$.
\begin{lemma}   \label{lemma-f_4}
For any  $k > 0$, there exists a solution $f_4(\rho, k)$
to \eqref{lin-3} such that $f_4= e^{k \rho}\chi_4$
with $\chi_4$ verifying 
$$\partial_\rho^l (\chi_4(\rho,k)-{0\choose1})=O_k(\rho^{-3-l}),\quad \rho\rightarrow +\infty.$$
\end{lemma}

\begin{proof}
We construct $f_4$ by means of the following integral equation:
\begin{equation}    \label{f-4}
\begin{split}
\chi_4(\rho, k) =& {0\choose1}
+\int_{\rho}^{+\infty} \left( \begin{array}{cc}
0 & 0\\
0 & \frac{1}{2 k}
\end{array} \right) V \chi_4(s, k) ds\\
&+\int_{R_1}^{\rho} \left( \begin{array}{cc}
\frac{e^{k(s-\rho)}\sin k (\rho -s)}{k} & 0\\
0 & \frac{e^{2k (s - \rho)}}{2k}
\end{array} \right) V\chi_4(s, k) ds.
\end{split}
\ee
For $k>0$ and $R_1$ sufficiently large (depending on $k$), the operator generating
\eqref{f-4} is small on the space of bounded continuous functions. Therefore, 
\eqref{f-4} has a solution $\chi_4$ verifying $|\chi_4(\rho, k)|\leq C$,
$\rho\geq R_1$. Iterating this bound one gets that
$\chi_4(\rho,k)-{0\choose 1}=O_k(\rho^{-3})$ as $\rho \rightarrow \infty$.
Finally, the estimates for the derivatives can be obtained  by differentiating \eqref{f-4}.
\end{proof}

We now briefly describe some properties  of the solutions $f_j$, $j=1,\dots, 4,$ that we will need later.
Recall that the Wronskian 
$w(f,g) = \left<f^\prime, g\right>_{\R^2} -\left<f, g^\prime\right>_{\R^2}$ does not depend on $\rho$ if 
$f$ and $g$ are solutions of \eqref{lin-2}. 

The estimates of Lemmas \ref{lemma-f_3}, 
\ref{lemma-f_1}, \ref{lemma-f_4}
lead to the relations:
\be\label{w-1}
w(f_1,f_2)=2ik,\,\,\, w(f_1,f_3)= w(f_2,f_3)=0,\,\,\, w(f_3,f_4)=-2k,\quad k>0,
\ee
the three first relations being valid  for $k=0$ as well.
Notice also that by Lemmas \ref{lemma-f_3},  \ref{lemma-f_1},
$\partial_kf_1(\rho,0)$, $\partial_kf_3(\rho,0)$ are solutions of 
the equation $\tilde Hf=0$  verifying for $\rho \geq 0$,
\begin{equation*}\begin{split}
&\big| \partial_{k} f_1(\rho, 0) - {i\rho\choose 0}\big|
 \leq C, 
\quad \big| \partial_{k \rho}^2 f_1(\rho, 0) - {i\choose 0}\big|
\leq \frac{C}{<\rho>^2},\\
&\big| \partial_{k}f_3(\rho, 0) +{0\choose \rho}\big| \leq \frac{C}{<\rho>}, \quad \big|
 \partial_{k \rho}^2 \zeta_3(\rho, 0) +{0\choose 1}\big| \leq \frac{C}{<\rho>^2},
\end{split}
\end{equation*}
\noindent As a consequence, one has
\be\label{w-2}\begin{split}
&w(\partial_k f_1|_{k=0},f_1|_{k=0})=i,\quad w(\partial_k f_1|_{k=0},f_3|_{k=0})=0,\\
&w(\partial_k f_3|_{k=0},f_1|_{k=0})=0,\quad w(\partial_k f_3|_{k=0},f_3|_{k=0})=-1.
\end{split}
\ee

In addition to scalar Wronskian we will use matrix Wronskians.
If $F$, $G$ are $2\times 2$ matrix solutions of \eqref{lin-3},
their matrix Wronskian
$$W(F,G)={F^t}^\prime G-{F^t} G^\prime$$
is independent of $\rho$.

Set $g_j(\rho, k)=f_j(-\rho, k)$, $j=1,\dots, 4$. Since the potential $V$ is even, 
$g_j$, $j=1,\dots, 4,$ are again solutions of \eqref{lin-3} which have the same asymptotic behavior as 
$\rho\rightarrow -\infty$ as $f_j$ as $\rho \rightarrow +\infty$.

Consider the matrix solutions $F$, $G$, defined by
$$F=(f_1,f_3),\quad G=(g_1, g_3).$$
Denote $D(k)=W(F,G)$. It follows from Lemmas \ref{lemma-f_3},  \ref{lemma-f_1}
that $D$ is smooth  for $k>0$ and admits the estimate
\be\label{D-1}
|\partial_k^2D(k)|\leq C \ln \left(\frac1k+1\right),\quad 0<k\lesssim 1.
\ee
In addition, by \eqref{k=0.1}, \eqref{w-1}, \eqref{w-2}, one has
\be\label{D-2}
D(0)=0,\quad \partial_kD(0)=\left( \begin{array}{cc}
-2i & 0\\
0 & 2
\end{array} \right).
\ee
\subsection{Scattering solutions and the distorted Fourier transform in a vicinity
of zero energy}
Set 
\be\label{F}
\mathcal F(\rho,k)=F(\rho,k)s(k),\ee
where $s(k)={D^t}^{-1}(k){2ik\choose 0}$.
By \eqref{D-1}, \eqref{D-2}, $s={s_1\choose s_2}$ is a smooth function of $k$ for $0<k<k_0$
($k_0$ sufficiently small), continuous up to $k=0$,  verifying
\be\label{s}\begin{split}
&s_1(0)=-1,\quad s_2(0)=0,\\
&|\partial_ks(k)|\leq C|\ln k|,\quad 0<k\leq k_0.
\end{split}
\ee
By construction, one has
$$w(\mathcal F, g_1)=2ik, \quad w(\mathcal F, g_3)=0,$$
for any $0\leq k<k_0$.
As a consequence,
\be\label{F-1}
\mathcal F(\rho,k)=r_1(k)g_1(\rho,k)+g_2(\rho,k)+r_2(k)g_3(\rho,k), \quad  0\leq k< k_0,
\ee
with some coefficients $r_1(k)$, $r_2(k)$ that, by \eqref{k=0.1}, \eqref{s}, verify
\be\label{r0}
r_1(0)=r_2(0)=0.
\ee
 Computing the Wronskians $w(\mathcal F,\bar{\mathcal F})$
and $ w(\mathcal F,\bar{\mathcal G})$, where $\mathcal G(\rho,k)=\mathcal F(-\rho,k)$,
one gets
$$|s_1(k)|^2+|r_1(k)|^2=1,\quad r_1(k)\overline{s_1(k)}+ \overline{r_1(k)}s_1(k)=0,
\quad  0\leq k< k_0.$$
One can write the following  Wronskian representation for $r_1$:
\be\label{r-1}
r_1(k)=s_1(k)\frac{w(g_2,f_1)}{2ik}+s_2(k)\frac{w(g_2,f_3)}{2ik},\quad k\neq 0.
\ee
Using \eqref{s} and the relations 
$$w(g_2,f_3)|_{k=0}=w(g_2,f_1)|_{k=0}=\partial_kw(g_2,f_1)|_{k=0},$$
one easily deduces from \eqref{r-1}  that $r_1$ is smooth for $0<k<k_0$, continuous up to $k=0$,
and verifies
\be\label{r1}
|\partial_kr_1(k)|\leq C|\ln k|,\quad 0<k< k_0,\ee
which in its turn, implies that $r_2$ is smooth for $0<k< k_0$, continuous up to $k=0$ and
admits a similar estimate:
\be\label{r2}
|\partial_kr_2(k)|\leq C|\ln k|,\quad 0<k< k_0.
\ee

Introduce the following odd solution of \eqref{lin-3}:
$$e(\rho,k)=\mathcal F(-\rho,k)-\mathcal F(\rho,k).$$
By \eqref{F}, \eqref{F-1},
\be\label{e}
e=a_1f_1+f_2+a_2f_3,\quad a_j=r_j-s_j, \,\,\, j=1,2.
\ee
It follows from \eqref{s}, \eqref{r0}, \eqref{r1},  \eqref{r2} that
\be\label{e-0}
a_1(0)=1,\,\,\, a_2(0)=0,
\ee
and 
\be\label{e-00}
|\partial_ka_j|\leq C|\ln k|,\quad 0<k< k_0,\,\,j=1,2,
\ee
which together with Lemmas  \ref{lemma-f_3},  \ref{lemma-f_1} 
implies the following result.
\begin{lemma}\label{lem-e} One has:\\
(i) $e(\rho,k)=e_0(\rho,k)+e_1(\rho,k)$, where
$e_0(\rho,k)=a_1(k)e^{ik\rho}{1\choose 0}+e^{-ik\rho}{1\choose 0}$ and the remainder
$e_1(\rho,k)$ admits the estimates
\be\label{e-1}\begin{split}
&|e_1(\rho,k)|\leq C(<\rho>^{-2}+k|\ln k|e^{-k\rho}),\quad \rho\geq 0,\\
&|\partial_ke_1(\rho,k)|\leq C|\ln k|(<\rho>^{-1}+e^{-k\rho/2}),\quad \rho\geq 0,\\
&\|e_1(\cdot, k)\|_{L^2(\R_+)}\leq C,\\
&\|\rho e_1(\cdot, k)\|_{L^2(\R_+)}+\|\partial_ke_1(\cdot, k)\|_{L^2(\R_+)}\leq Ck^{-1/2}|\ln k|,
\end{split}
\ee
for any $0<k\leq k_0$.\\
(ii) $(\rho\partial_\rho -k\partial_k)e(\rho,k)=e^{ik\rho}{1\choose 0}k\partial_ka_1(k)+e_2(\rho,k)$,
with $e_2(\rho,k)$ verifying
\be\label{e-2}\begin{split}
&|e_2(\rho,k)|\leq C(<\rho>^{-1}+k|\ln k|e^{-k\rho/2}),\quad \rho\geq 0,\\
&\|e_2(\cdot, k)\|_{L^2(\R_+)}\leq C,\\
\end{split}
\ee
for any $0<k\leq k_0$.
\end{lemma}

For $0<\kappa\leq k_0$, introduce the operators $\mathbb E_\kappa: L^2(\R_+,\C^2)\rightarrow
 L^2(\R^3,\C^2)$,
$$(\mathbb E_\kappa\Phi)(y)=\frac{1}{2^{3/2}\pi}\int_{\R_+}dk\theta_\kappa(k)\mathcal E(y,k)\Phi(k),\quad \Phi
\in L^2(\R_+,\C^2), $$
where $\mathcal E(y,k)$ is a $2\times 2$ matrix given by
$$\mathcal E(y,k)=\rho^{-1}(e(\rho,k),\sigma_1\overline{e(\rho,k)}),\quad \rho=|y|,$$
$\theta_\kappa(k)=\theta(\kappa^{-1}k)$, $\theta$ is a $C^\infty$ even function verifying
 $\theta(k) = \left\{ \begin{array}{ll}
1 & \textrm{if} \,\, |k| \leq 1/4\\
0 & \textrm{if} \,\, |k | \geq 1/2
\end{array} \right.$.

Since $e(\rho,k)$ is a solution of  the equation $\tilde He=k^2e$, one has
$H\mathbb E_\kappa=\mathbb E_\kappa k^2\sigma_3$.

By Lemma \ref{lem-e} (i), the operators $\mathbb E_\kappa$ are bounded uniformly with 
respect to $\kappa\leq k_0$. 
The action of the adjoint operators $\mathbb E^*_\kappa:
 L^2(\R^3,\C^2)\rightarrow L^2(\R_+,\C^2)$
 is given by
$$(\mathbb E^*_\kappa\psi)(k)=
\frac{1}{2^{3/2}\pi}\theta_\kappa(k)\int_{\R^3} dy \mathcal E^*(y,k)\psi(y),\quad \psi \in  L^2(\R^3,\C^2).$$
Clearly, 
\be\label{comp-0}
\mathbb E^*_\kappa\sigma_3\zeta_\pm=0
\ee for any $0<\kappa\leq k_0$.

The following relation is 
 a standard consequence of the asymptotics given by  Lemma \ref{lem-e} (i),
\be\label{comp}
\mathbb E^*_{\kappa_2}\sigma_3\mathbb E_{\kappa_1}\sigma_3=\theta_{\kappa_1}(k)\theta_{\kappa_2}(k),
\ee
for any $0<\kappa_1,\kappa_2\leq k_0$.
\begin{remark}
Notice that because of the presence of the cut off function $\theta_\kappa$,
$\mathbb E_\kappa$ is bounded as an operator from $L^2([0, k_0])$ to $H^m(\R^3)$
for any $m\geq 0$, uniformly in $\kappa\leq k_0$.
\end{remark}

We next introduce quasi-resonant functions $h_\kappa(y)$, $0<\kappa\leq k_0$, by setting
$$h_\kappa=\sqrt 2\mathbb E_\kappa{1\choose 0}.$$
\begin{lemma}
For any $0<\kappa \leq k_0$, $h_\kappa\in <y>^{-1}L^2(\R^3)$ and as $\kappa\rightarrow 0$, one has
\begin{equation}\label{h-00}
\|h_\kappa\|_{L^2(\R^3)}=O(\kappa^{1/2}),\quad \|yh_\kappa\|_{L^2(\R^3)}=O(\kappa^{-1/2}),\ee
\be\label{h-000}
\left<h_\kappa,\sigma_3(\xi_0+\xi_1)\right>=4\pi+O(\kappa^{1/2}\ln \kappa),\quad
\left<h_\kappa,\sigma_3(\xi_1-\xi_0)\right>=O(\kappa^{1/2}\ln \kappa).
\ee
\end{lemma}
\begin{proof}
Applying Lemma \ref{lem-e} (i), we decompose $h_\kappa$ as follows:
\begin{equation}\label{h-0}\begin{split}
&h_\kappa(y)=h_{\kappa,0}(y)+h_{\kappa,1}(y)+h_{\kappa,2}(y),\\
&h_{\kappa,0}(y)=\frac{1}{2\pi \rho}\kappa\hat \theta(\kappa \rho){1\choose 0},\\
&h_{\kappa,1}(y)=\frac{1}{2\pi \rho}\int_{\R_+}dk e^{ik\rho}(a_1(k)-1)\theta_\kappa(k){1\choose 0},\\
&h_{\kappa,2}(y)=\frac{1}{2\pi \rho}\int_{\R_+}dk \theta_\kappa(k)e_1(\rho,k),
\end{split}
\end{equation}
where $\hat \theta(\rho)=\int_\R e^{ik\rho}\theta(k) dk$, $\rho=|y|$.

Clearly, $h_{\kappa,0}\in <y>^{-1}L^2(\R^3)$ and one has
\be\label{h-01}
\|h_{\kappa,0}\|_{L^2(\R^3)}\leq C\kappa^{1/2},\quad \|yh_{\kappa,0}\|_{L^2(\R^3)}\leq C\kappa^{-1/2}.
\ee
Consider $h_{\kappa,i}$, $i=1,2$.
It follows from  \eqref{e-0}, \eqref{e-00}, \eqref{e-1} that
\be\label{h-02}
\|h_{\kappa,i}\|_{L^2(\R^3)}\leq C \kappa,\,\,\|yh_{\kappa,i}\|_{L^2(\R^3)}\leq C \kappa^{1/2}|\ln\kappa|,\quad i=1,2,
\ee
which together with \eqref{h-01} leads to the estimates
\be\label{h-03}
\|h_{\kappa}\|_{L^2(\R^3)}\leq C\kappa^{1/2},\quad \|yh_{\kappa}\|_{L^2(\R^3)}\leq C\kappa^{-1/2}.
\ee
We next compute $\left<h_\kappa,\sigma_3(\xi_1\pm \xi_0)\right>$. By  \eqref{h-01}, \eqref{h-02},
as $\kappa\rightarrow 0$, one has
\begin{equation}\label{h-2}\begin{split}
&\left<h_\kappa,\sigma_3(\xi_1\pm \xi_0)\right>=\left<h_{\kappa,0},\sigma_3(\xi_1\pm \xi_0)\right>+O(\kappa^{1/2}\ln \kappa),\\
&\left<h_{\kappa,0},\sigma_3(\xi_1-\xi_0)\right>=O(\kappa),\\
&\left<h_{\kappa,0},\sigma_3(\xi_1+\xi_0)\right>={2\kappa}\int_\R  d\rho
\hat \theta(\kappa\rho) +O(\kappa)=4\pi+O(\kappa),
\end{split}
\ee
which gives  \eqref{h-000}.
\end{proof}

\subsection{Proof of Proposition \ref{lin}}
We start by deriving some coercivity bounds for the operator $H$.
\begin{lemma}\label{4-1}
There exists $\kappa_0$, $0<\kappa_0\leq k_0$, and $C>0$ such that
\be\label{coer-1}
\left<Hf,\sigma_3 f\right>\geq C\kappa \|\nabla f\|_{L^2(\R^3)}^2,
\ee
for any $0<\kappa\leq \kappa_0$ and any $f\in\dot H^1_{rad}(\R^3,\C^2)$ verifying
\be\label{orth-1}
\left<f,\sigma_3 \zeta_-\right>=\left<f,\sigma_3 \zeta_+\right>=\left<f,\sigma_3 h_\kappa\right>=
\left<f,\sigma_3 \sigma_1\bar h_\kappa\right>=0.
\ee
\end{lemma}
\begin{remark}
Notice  that since $\zeta_\pm, h_\kappa\in <y>^{-1}L^2(\R^3)$ the scalar products that
appear in \eqref{orth-1} are well defined for any $f\in \dot H^1$.
\end{remark}
\begin{proof}
The proof of Lemma \ref{4-1} is based on the following result which is due to 
Duyckaerts and Merle:
\begin{lemma}\label{DM}
There exists  $c_0>0$ such that
\begin{equation*}
\left<Hf,\sigma_3 f\right>\geq c_0 \|\nabla f\|_{L^2(\R^3)}^2,
\end{equation*}
for any $f\in\dot H^1_{rad}(\R^3,\C^2)$ verifying
\begin{equation*}
\left<f,\sigma_3 \zeta_-\right>=\left<f,\sigma_3 \zeta_+\right>=\left<f,\Delta \xi_0\right>=
\left<f,\Delta \xi_1\right>=0,
\end{equation*}
\end{lemma}
see \cite{DM} for the proof.

Let $f\in \dot H^1_{rad}$ such that \eqref{orth-1} holds. One can write $f$ as
$$f=\alpha_0 \xi_0+\alpha_1 \xi_1+g,$$
where
$$\alpha_j=-\frac{\left <f,\Delta \xi_j\right>}{\|\nabla \xi_j\|_{L^2(\R^3)}^2},\quad j=0,1,
$$
and $g\in \dot H^1_{rad}$ verifies
$$\left<g,\sigma_3 \zeta_-\right>=\left<g,\sigma_3 \zeta_+\right>=\left<g,\Delta \xi_0\right>=
\left<g,\Delta\xi_1\right>=0.
$$
Therefore, by Lemma \ref{DM},
\be\label{DM-1}
\left<H f, \sigma_3 f\right>= \left<H g, \sigma_3 g\right>\geq  c_0\| \nabla g \|_{L^2(\R^3)}^2.
\ee Furthermore, since $f$ verifies \eqref{orth-1}, one has
$$A(\kappa) {\alpha_0\choose \alpha_1}=
{\left<g, \sigma_3 h_{ \kappa}\right>\choose\left<g, \sigma_3 \sigma_1\bar h_{ \kappa}\right>},$$
where 
$$A(\kappa)=-\left( \begin{array}{cc}
\big < \xi_0,  \sigma_3 h_{\kappa}\big>&\left < \xi_1,  \sigma_3 h_\kappa\right>\\
\left < h_{\kappa},\sigma_3\xi_0\right>&-\left < h_\kappa,\sigma_3\xi_1\right>
\end{array} \right).$$
By \eqref{h-000},
$$A(\kappa) = -2\pi\left( \begin{array}{cc}
1 & 1\\
1 & -1
\end{array} \right)+O(\kappa^{1/2}\ln\kappa),\quad \kappa\rightarrow 0.$$
 Therefore, for $\kappa$ sufficiently small, one has
$$|\alpha_0| +|\alpha_1| \leq C \| \nabla g \|_{L^2(\R^3)} \| <y> h_{\kappa} \|_{L^2(\R^3)}
\leq C \kappa^{-1/2} \| \nabla g \|_{L^2(\R^3)}.
$$
As a consequence,
$$\| \nabla f \|_{L^2(\R^3)} \leq C \kappa^{-1/2} \| \nabla g \|_{L^2(\R^3)}.$$
Combining this inequality with \eqref{DM-1} we get \eqref{coer-1}.
\end{proof}

Next, we prove
\begin{lemma}	\label{grad-estim}
There exists $\kappa_1$, $0<\kappa_1\leq k_0$, and $C>0$ such that
for any $0<\kappa\leq \kappa_1$ one has
$$\| f\|_{H^1(\R^3)} \leq \frac C\kappa \| \nabla f \|_{L^2(\R^3)},$$
for all  $f \in H^1_{rad}(\R^3)$ verifying
$\mathbb{E}_{\kappa}^*f= 0$.
\end{lemma}
\begin{proof}
By \eqref{e-0}, \eqref{e-00} and Lemma \ref{lem-e} (i), $\mathbb{E}_{\kappa}^*f$
can be written as
$$(\mathbb{E}_{\kappa}^*f)(k)=\Phi_0(k)+\Phi_r(k),$$
where
$$\Phi_0(k)=\frac{1}{2^{3/2}\pi}\theta_\kappa(k)\check f(k),$$
$\check f(k)=2\int_{\R^3}dy\frac{\cos k|y|}{|y|}f(y) $, 
and the remainder $\Phi_r$ satisfies
$$
\|\Phi_r\|_{L^2(\R_+)}\leq C\kappa^{1/2}\| f \|_{L^2(\R^3)}.
$$
Therefore, $\mathbb{E}_{\kappa}^*f= 0$ implies
\be\label{g}
\|\check f\|_{L^2(0,\kappa/4)}\leq C\kappa^{1/2}\| f \|_{L^2(\R^3)}.
\ee
Notice also that for any $f\in H^1_{rad}$ and any $0<\kappa\leq 1$ one has
$$\| f\|_{H^1(\R^3)} \leq C(\|\check f\|_{L^2(0,\kappa/4)}+\kappa^{-1} \| \nabla f \|_{L^2(\R^3)}).$$
Combining this inequality with \eqref{g}, we get
$$\| f\|_{H^1(\R^3)} \leq \frac C\kappa \| \nabla f \|_{L^2(\R^3)},$$
provided $\kappa$ is sufficiently small.
\end{proof}
We finally combine Lemmas \ref{4-1}, \ref{grad-estim} to derive the following
result which will be in the heart of the proof of Prop. \ref{lin}

\begin{lemma}	\label{H-estimates}
There exists $\kappa_2$,
$0<\kappa_2\leq k_0$, and $C>0$ such that for any $0 < \kappa \leq \kappa_2$ one has 
\be\label{w0}
\left< H f, \sigma_3 f \right> \geq C\kappa^3\| f\|_{H^1}^2 -\frac \kappa C\| \mathbb{E}_{\kappa}^*\sigma_3f \|_{L^2(\R_+)}^2,
\ee
 for any $f\in H^1_{rad}(\R^3,\C^2)$ verifying $\left<f, \sigma_3 \zeta_{\pm}\right> = 0$.
\end{lemma}

\begin{proof}
Write $f= f_1 +f_2$, where  $f_1 = \mathbb{E}_{\kappa}\sigma_3\mathbb{E}_{\kappa}^*\sigma_3f $
and $f_2=f- f_1$. One clearly has
\be\label{ww}\| f_1 \|_{H^1(\R^3)} \leq C \| \mathbb E_{\kappa}^*\sigma_3 f\|_{L^2(\R_+)} ,\quad
\| H f_1 \|_{L^2(\R^3)} \leq C \kappa^2 \| \mathbb E_{\kappa}^* \sigma_3 f\|_{L^2(\R_+)},
\ee
for any $0<\kappa\leq k_0$.

Consider $f_2$. It follows from \eqref{comp-0},  \eqref{comp} that for any $\kappa^\prime\leq \kappa/2$,
\begin{itemize}
\item[$\bullet$] $\left<f_2, \sigma_3 \zeta_{\pm}\right>= 0$;
\item[$\bullet$] $\mathbb{E}_{\kappa^\prime}^* \sigma_3f_2 = 0$;
\item[$\bullet$] $\left<f_2, \sigma_3 h_{ \kappa^\prime}\right>=\left<f_2, \sigma_3\sigma_1\bar h_{ \kappa^\prime}\right> =0$.
\end{itemize}
Hence, by Lemmas  \ref{4-1}, \ref{grad-estim}, one has
\be\label{www}
\left<H f_2, \sigma_3 f_2\right>\geq C \kappa ^3\| f_2 \|_{H^1(\R^3)}^2,
\ee
provided $\kappa$ is sufficiently small.

Combining  \eqref{ww}, \eqref{www} one gets \eqref{w0}.

\end{proof}

\noindent We are now in the position to prove Proposition \ref{lin}.
Consider the  equation
\be\label{lin-5}\begin{split}
&i\psi_\tau=P\mathcal H(\tau)P\psi, \\
&\psi|_{\tau=s}= f,
\end{split}
\ee
where
$$\mathcal H(\tau)=H+\tau^{-1}l,\quad l=\alpha_1\sigma_3-i\nu_1(\frac12 +y\cdot \nabla),$$
$\alpha_1,\nu_1\in \R$, $s>0$ and $f\in \mathcal S(\R^3)$ verifying
$\left<f, \sigma_3 \zeta_{\pm}\right> = 0$.

Fix $\kappa$ such that
$0<\kappa\leq \kappa_2$  and consider 
the functional 
$G_1(\tau)=\left<H \psi, \sigma_3 \psi\right>+
c_0\|\mathbb E_\kappa^*\sigma_3\psi\|_{L^2(\R_+)}^2$.
Clearly,
\be\label{end-00}G_1(\tau)\leq C\|\psi(\tau)\|_{H^1(\R^3)}^2.
\ee
Moreover,
since $\left<\psi(\tau), \sigma_3 \zeta_{\pm}\right>=0 $, choosing $c_0$ sufficiently large, we get:
\be\label{end-1}
G_1(\tau)\geq c_1\|\psi(\tau)\|_{H^1(\R^3)}^2.
\ee
We next compute the derivative $\frac{d}{d\tau}G_1$.
One has
$$i\frac{d}{d\tau}\left<H \psi, \sigma_3 \psi\right>=\frac{2i}{\tau}\Im \left<l\psi, \sigma_3 H\psi\right>
,$$
which implies
\be\label{end-2}
\left|\frac{d}{d\tau}\left<H \psi, \sigma_3 \psi\right>\right|\leq \frac C\tau(|\alpha_1|+|\nu_1|)
\|\nabla\psi(\tau)\|_{L^2(\R^3)}^2.
\ee
Next, we address $\|\mathbb E_\kappa^*\sigma_3\psi\|_{L^2(\R_+)}^2$.
Denote $\Phi(\tau)=\mathbb E_\kappa^*\sigma_3\psi(\tau)$. Then $\Phi(k,\tau)$ solves
\be\label{end-3}
i \Phi_{ \tau} = k^2 \sigma_3\Phi +\frac1\tau Y,
\ee
where
$$Y=\mathbb E_\kappa^*\sigma_3l\psi.$$
Integrating by parts and
applying Lemma \ref{lem-e} (ii), one can rewrite $Y$ in the form
$$Y(k,\tau)=Y_0(k,\tau)+Y_1(k,\tau),$$
where 
$$Y_0(k,\tau)=i\nu_1 k\partial_k\Phi(k,\tau),$$
and $Y_1(k,\tau)$ admits the estimate
$$\|Y_1(\tau)\|_{L^2(\R_+)}\leq C(|\alpha_1|+|\nu_1|)\|\psi(\tau)\|_{L^2(\R^3)}.$$
Therefore, \eqref{end-3} gives
$$\left|\frac{d}{d\tau}\|\Phi(\tau)\|_{L^2(\R_+)}^2\right|\leq \frac C\tau (|\alpha_1|+|\nu_1|)\|\psi(\tau)\|_{L^2(\R^3)}^2.$$
Combining this inequality with \eqref{end-3} and taking into account \eqref{end-1}
one gets
\be\label{end-4}
\left|\frac{d}{d\tau}G_1(\tau)\right|\leq \frac C\tau (|\alpha_1|+|\nu_1|)\|\psi(\tau)\|_{H^1(\R^3)}^2\leq \frac C\tau (|\alpha_1|+|\nu_1|)G_1(\tau).
\ee
Integrating we obtain
$$G_1(\tau)\leq C\left(\frac s\tau\right)^{C (|\alpha_1|+|\nu_1|)}G_1(s),\quad 0<\tau\leq s,$$
which by \eqref{end-00},  \eqref{end-1}, leads to the bound
\be\label{end-7}
\|U(\tau,s)f \|_{H^1(\R^3)} \leq C \left( \frac{s}{\tau} \right)^{C(|\alpha_1|+|\nu_1|)} \| f\|_{H^1(\R^3)},
\ee
for any $0<\tau\leq s$ and any $f\in H_{rad}^1$.

To control the higher regularity, consider the functional
$G_2(\tau)=\left<H^2 \psi, \sigma_3 H\psi\right>+c_2 G_1(\tau)$.
One has 
$$C^{-1}\|\psi\|_{H^3(\R^3)}^2\leq  G_2\leq C\|\psi\|_{H^3(\R^3)}^2,$$
provided $c_2$ is chosen sufficiently large.

Computing the derivative $\frac{d}{d\tau}\left<H^2 \psi(\tau), \sigma_3 H\psi(\tau)\right>$
and taking into account \eqref{end-4}
we get
\be\label{end-5}
\left|\frac{d}{d\tau}G_2(\tau)\right|\leq \frac C\tau (|\alpha_1|+|\nu_1|)\|\psi(\tau)\|_{H^3(\R^3)}^2
\leq \frac C\tau (|\alpha_1|+|\nu_1|)G_2(\tau).
\ee
which implies
\be\label{end-8}
\|U(\tau,s)f \|_{H^3(\R^3)} \leq C \left( \frac{s}{\tau} \right)^{C(|\alpha_1|+|\nu_1|)} \| f\|_{H^3(\R^3)},\ee
for any $0<\tau\leq s$.

The $H^2$ bounded stated in Prop. \ref{lin} follows from \eqref{end-7}, \eqref{end-8} by interpolation.

\end{document}